\documentclass[12pt, a4paper]{article}

\usepackage[margin=2cm]{geometry}
\usepackage{graphicx}
\usepackage{color}
\usepackage{epstopdf}
\usepackage{graphicx}
\usepackage{algorithmicx}
\usepackage{booktabs}
\usepackage{amsmath, amsfonts, amssymb, latexsym}
\usepackage{mathtools}
\usepackage{algorithm}
\usepackage{algpseudocode}
\usepackage{amsmath}
\usepackage{xcolor}
\usepackage{epstopdf}
\usepackage{comment}
\usepackage{upref}
\usepackage{hyperref}
\usepackage{fancyvrb}
\usepackage{amsmath,amsfonts,amssymb}
\usepackage{mathtools}
\usepackage{float}
\usepackage{subfig}
\usepackage{commath}
\usepackage{graphicx}
\usepackage{subcaption}
\usepackage{subfig}
\usepackage{placeins}

\newcommand{\What}{\widehat{W}}
\newcommand{\Shat}{\widehat{S}}
\newcommand{\Vhat}{\widehat{V}}

\newcommand{\Lhat}{\widehat{L}}
\newcommand{\Phat}{\widehat{P}}
\newcommand{\Qhat}{\widehat{Q}}
\newcommand{\That}{\widehat{T}}

\newcommand{\E}{\theta}
\newcommand{\F}{\phi}

\def\norm#1{\|#1\|}

\usepackage{amsthm}
\theoremstyle{remark}

\newtheorem{theorem}{Theorem}[section]
\newtheorem{lemma}[theorem]{Lemma}
\newtheorem{remark}[theorem]{Remark}

\begin{document}
\title{Error Analysis and Precision Selection for Mixed-Precision DEIM-CUR Decompositions}
\author{Ioannis Thanasis\thanks{Faculty of Mathematics and Physics, Charles University, \url{thanasis@karlin.mff.cuni.cz}}, Erin Carson\thanks{Faculty of Mathematics and Physics, Charles University, \url{carson@karlin.mff.cuni.cz}\\Both authors are supported by the European Union (ERC, inEXASCALE, 101075632).
Views and opinions expressed are those of the authors only and do not necessarily reflect those of the European Union or the European Research Council.
Neither the European Union nor the granting authority can be held responsible for them.
The second author additionally acknowledges support from the Charles University Research Centre program No. UNCE/24/SCI/005.}}
\date{}
\maketitle

\begin{abstract}
CUR approximation is a popular approach for constructing an interpretable low-rank approximation of a matrix. There are many approaches for constructing the CUR approximation, many of which are based on the SVD. Results in the literature bounding the quality of the approximation usually rely on the assumption that all or part of the SVD is computed exactly. In this work, we relax this assumption and show how the quality of the approximation is affected by an inexact SVD. Moreover, we analyze the effect of finite precision arithmetic on the DEIM-CUR method. Our analysis suggests guidelines for selecting the working precisions for the SVD and the DEIM index selection algorithm. In particular, it suggests that in many cases one can use very low precisions for these computations without significantly impacting the approximation quality. A set of numerical experiments is presented to demonstrate the theoretical results. 
\end{abstract}

\section{Introduction}
\label{sec:intro}

\hspace{5mm}The low-rank approximation of matrices is a fundamental task in scientific computing, data science, and machine learning applications. While the truncated singular value decomposition (SVD) provides the optimal rank-$k$ approximation of a matrix $A$ in the $2$-norm and the Frobenius norm, it has undesirable properties in certain applications, namely, that the resulting factors do not retain properties of $A$ such as sparsity and nonnegativity. 

In this work, we consider the alternative interpretable decompositions, in particular, the CUR decomposition, which approximates
\begin{equation*}
A \approx CUR,
\end{equation*}
where for $A\in\mathbb{R}^{m\times n}$, $C\in\mathbb{R}^{m\times k}$ consists of $k$ columns of $A$, $R\in\mathbb{R}^{k\times n}$ consists of $k$ rows of $A$, and $U \in \mathbb{R}^{k \times k}$ is constructed to make $CUR$ close to $A$. Finding a set of columns and rows to minimize $\Vert A-CUR\Vert$ is an instance of the column subset selection problem, which is NP-Hard. Thus a number of heuristic approaches exist, usually either based on QR or SVD.

Here we consider one approach, the Discrete Empirical Interpolation Method (DEIM), originally developed for model order reduction of nonlinear dynamical systems \cite{chaturantabut2010nonlinear}, which uses a greedy approach based on interpolatory projection to select indices \cite{sorensen2016deim}. \par

Modern computing hardware increasingly supports a variety of numerical precisions, ranging from 4-bit microscaling formats to IEEE double precision (fp64). Lower precision formats offer significant advantages in computational speed, memory bandwidth, and energy efficiency. On today's GPUs, these advantages are nonlinear; for example, on the upcoming NVIDIA Rubin architecture, 8-bit formats have a theoretical peak performance of over 500$\times$ that of 64-bit formats \cite{nvidia_vera_rubin_nvl72_2026}.
It is thus desirable to make use of low precision formats as much as possible, without introducing unacceptable levels of error or loss of stability.

Despite much work in developing reduced and mixed precision approaches in numerical linear algebra, see, e.g., \cite{higham2022mixed} and the references therein, little work has been done on low and mixed precision strategies for interpretable decompositions. The  DEIM-CUR computation consists of the computation of an initial truncated SVD, followed by what is essentially an LU factorization with partial pivoting. This thus raises the question: which parts of the computation can we compute in lower precision while maintaining acceptable accuracy of the resulting CUR approximation? Further, can low precision be combined with a randomized approach to computing the SVD to further reduce the cost? 

Existing error analysis for DEIM-CUR, see, e.g., \cite{sorensen2016deim, drmac2016new}, largely assumes exact arithmetic, or at least orthogonality of certain factors in the SVD or QR factorizations. These do not provide guidance on the precision requirements relative to the desired rank $k$ or target accuracy.

In this paper, we develop a rigorous finite precision error analysis of DEIM-induced CUR decompositions and propose guidelines which can be used to select appropriate precisions. Our main contributions are:
\begin{itemize}
    \item We derive a modified version of the fundamental lemmas and theorems of Sorensen and Embree \cite{sorensen2016deim} which remove the assumption that we have an exact SVD at hand. Our bounds are applicable to general inexact SVD (which could be computed, for example, using a randomized method);
    \item We derive bounds for the resulting amplification factors when the DEIM index selection algorithm itself is computed in finite precision;
    \item Our analysis suggests guidelines for selecting the precisions which should be used for the SVD and the DEIM index selection algorithm. 
\end{itemize}

The remainder of the paper is outlined as follows. In Section \ref{sec:background}, we give an overview of CUR decompositions and the DEIM algorithm for point selection. In Section \ref{sec:inexactSVD}, we investigate the quality of a CUR decomposition constructed using an inexact SVD, meaning that we may have only approximate singular values and we do not necessarily have exactly orthogonal singular vectors. These results hold for any method of point selection, where different methods of point selection will affect the included amplification factors. In Section \ref{sec:fpDEIM}, we consider the DEIM algorithm computed in finite precision and bound the resulting amplification factors, which can then be used within the analysis from Section \ref{sec:inexactSVD}. In Section \ref{sec:specSVD} we make the results of Section \ref{sec:inexactSVD} concrete for two example SVD algorithms: LAPACK's \texttt{GESVD} computed in finite precision and a randomized SVD with oversampling computed in finite precision. We demonstrate our theoretical results in Section \ref{sec:numexp} and summarize in Section \ref{sec:conclusion}.
\section{Background and Preliminaries}
\label{sec:background}

\subsection{CUR Decompositions}

\hspace{5mm}CUR decompositions are low rank matrix decompositions expressed in terms of a small subset of rows and columns of the data matrix. Formally, for a matrix A, let I and J denote row and column indices respectively. Then, $A \approx CUR$, where $C = A(:,J)$ and $R = A(I,:)$. Common choices for the core/intersection matrix $U$ are $U = C^\dagger A R^\dagger$ or $U = A(I,J)^\dagger$. Originating from ``pseudoskeleton" approximations \cite{goreinov1997theory}, various approaches have been proposed to select a suitable set of indices, which can broadly be divided into two categories.\par
Pivoting based methods, like column pivoted QR \cite{golub2013matrix} or LU with complete pivoting \cite{trefethen1997numerical}, are applied to the data matrix or its dominant singular vectors, where the pivots determine the selected row or column indices. Sorensen and Embree show in \cite{sorensen2016deim} that when the DEIM-CUR uses the dominant singular vectors, then the approximation error is bounded by a factor of the optimal approximation that depends on the conditioning of the submatrices of the dominant singular vectors. Although their analysis is stated for the DEIM-CUR method, their results are applicable to a general class of CUR factorizations. Hamm and Huang in \cite{hamm2021perturbations} investigate the effect on the approximation error in the presence of noise in the data for several variants. They give perturbation estimates in terms of the magnitude of the noise matrix and show how the choice of columns and rows affects the quality of the approximation. Park and Nakatsukasa \cite{park2025accuracy} investigate the accuracy and numerical stability of CUR decompositions. CUR decompositions can become unstable, due to the inverse of the nearly singular core matrix, causing an amplification of errors. Their finite precision analysis demonstrates that CUR decompositions can be implemented in a numerically stable manner with theoretical guarantees by selecting additional indices. \par
Sampling based methods select the indices by a random sampling according to a probability distribution based on the data. For example, in the case of leverage score sampling \cite{mahoney2009cur}, the distribution is based on the squared Euclidean norm  of each row of the dominant singular vectors. Such methods lead to probabilistic bounds on $\norm{A-CUR}_F$. Other sampling strategies include uniform sampling \cite{chiu2013sublinear}, volume sampling \cite{cortinovis2020low}\cite{deshpande2006adaptive}, DPP sampling \cite{derezinski2021determinantal} and BSS sampling \cite{boutsidis2014near}. 	\par

\subsection{DEIM Index Selection}
\hspace{5mm}DEIM was originally introduced as a Model Order Reduction (MOR) technique \cite{chaturantabut2010nonlinear} for high-dimensional dynamical systems, often arising from spatial discretizations of nonlinear PDEs. The method combines projection with interpolation in a way that improves the computational efficiency of evaluating the nonlinear term within a Proper Orthogonal Decomposition (POD) with Galerkin projection. The nonlinear term is approximated so that a coefficient matrix can be precomputed, keeping the complexity of its evaluation proportional to a small number of selected spatial indices. \par 
In the case of CUR factorization, the DEIM procedure introduced in \cite{sorensen2016deim} operates on the dominant singular vectors of a matrix $A$. It is applied separately to the left and right singular vectors, here denoted $V$ and $W$, respectively, to select indices I and J, such that $C = A(:,J)$ and $R = A(I,:)$ capture the dominant column and row spaces of the original matrix. Specifically, given the left singular vectors, the method processes those vectors one by one and greedily chooses the interpolation index corresponding to the largest magnitude entry of the interpolation residual vector.  Given the left singular vectors $V_j = [v_1, v_2, ..., v_j]$, the process selects the next index $p_j$ as 
\[
    r_j = v_j - \mathcal{P}_{j-1}v_j, \qquad   |r_j(p_j)| = ||r_j||_\infty,
\]
where $\mathcal{P}_j \equiv V_j(P^TV_j)^{-1}P^T$, $P \equiv I(:,\mathbf{p}_j)$ and $\mathbf{p}_j = [p_1, p_2, ..., p_j]$. Similarly, for the right singular vectors $W_j =[w_1, w_2, ..., w_j]$, we denote $\mathcal{Q}_j \equiv Q(W_j^T Q)^{-1}W_j^T$, $Q \equiv I(:, \mathbf{q}_j)$ and $\mathbf{q}_j = [q_1, q_2, ..., q_j]$. Algorithm \ref{alg:DEIM} describes the process in detail.\par
In \cite{sorensen2016deim} the formulation of the method is accompanied by an error analysis that is applicable to a broad class of CUR factorizations. They bound the approximation error by a factor of the optimal approximation that depends on the conditioning of the submatrices ($P^TV_k$ and  $W_k^TQ$) of the dominant singular vectors. Specifically: 
\begin{equation*}
    \norm{A-CUR}_2 \leq (\eta_p+\eta_q)\sigma_{k+1}(A),
\end{equation*}
where $\eta_p = \norm{(P^TV_k)^{-1}}_2$, $\eta_q = \norm{(W_k^TQ)^{-1}}_2$ and $\sigma_{k+1}(A)$ is the $(k+1)$th singular value of the matrix $A$ (in descending order). Moreover, they propose an incremental QR algorithm for approximating the dominant singular vectors that makes only one pass through the matrix A.  
\begin{algorithm}
    \caption{DEIM point selection algorithm}\label{alg:DEIM}
    \scalebox{1}{
    \begin{minipage}{\linewidth}
    \begin{algorithmic}[1]
    \Procedure{[$p$] = DEIM}{$V$}
        \State $\textit{\% $V$: $m\times k$ matrix ($m \geq k$)}$
        \State $\textit{\% $p$: an integer vector with k distinct entries in $\{1...m\}$}$
        \State $v = V(:,1)$
        \State $[\sim,p_1 ] = \max(|v|)$
        \State $\mathbf{p} = [p_1]$
        \For{ j = 2...k }   
            \State $v = V(:,j)$
            \State $c = V(\mathbf{p},1:j-1)^{-1}v(\mathbf{p})$
            \State $r = v - V(:,1:j-1)c$
            \State $[\sim,p_j] = \max(|r|)$
            \State $\mathbf{p} = [\mathbf{p}; p_j]$
        \EndFor
    \EndProcedure
    \end{algorithmic}
    \end{minipage}
    }
\end{algorithm}

\subsection{Notation}
\hspace{5mm}We let $\Vert \cdot \Vert$ denote the 2-norm and $\Vert \cdot \Vert_F$ denote the Frobenius norm. In general, we use hats ($\hat{\cdot}$) to denote quantities that are computed inexactly. The dagger symbol $\dagger$ denotes the pseudoinverse. In many cases we adopt the gamma notation for rounding error analysis, using the quantities
\[
\gamma_n = \frac{nu}{1-nu},\quad\text{and}\quad \widetilde{\gamma}_n = \frac{cnu}{1-cnu},
\]
assuming $nu<1$ and $cnu<1$, respectively, where $c$ is a small integer constant. When $\gamma$ carries a superscript, this denotes the subscript on the corresponding unit roundoff $u$; here we will use two units roundoff: $u_D$ for the DEIM point selection and $u_S$ for the SVD computation. 
\section{CUR Approximation with Inexact SVD}
\label{sec:inexactSVD}

\hspace{5mm}Let $A = V S W^T$ be the exact full SVD of $A$, and let $A_k = V_k S_k W_k^T$ be the best rank-$k$ approximation of $A$ given by the (exactly computed) truncated SVD of $A$, where $V_k \in \mathbb{R}^{m\times k}, S_k \in \mathbb{R}^{k \times k}$, and $W_k \in \mathbb{R}^{n \times k}$. 

Assume that we have computed a rank-$k$ truncated SVD of an $m\times n$ matrix $A$ inexactly in some way, and that the inexact SVD is given by $\widehat{A}_k = \Vhat_k \Shat_k \What_k^T$. Thus $\Shat_k$ does not contain the exact leading $k$ singular values, and we do not assume that $\Vhat_k$ and $\What_k$ have orthonormal columns. 
We assume that we can write the bound
 \begin{equation}
 \norm{A-\Vhat_k \Shat_k \What_k^T} \leq \E \sigma_{k+1} +\F,
 \label{eq:inexactSVD}
 \end{equation}
for some quantities $\E$ and $\F$. The reason for this format will become clear when we later apply our results to specific inexact SVD computations. Note that for the exact SVD computation, we have $\E=1$ and $\F=0$.

 We consider a CUR factorization that uses row indices $p\in\mathbb{N}^k$ and column indices $q\in\mathbb{N}^k$, and set
\begin{equation}
P=I(:,p)\in\mathbb{R}^{m\times k}, \quad Q=I(:,q)\in\mathbb{R}^{n\times k}.
\label{eq:PQdef}
\end{equation}
We first prove an analog of \cite[Lemma 4.1]{sorensen2016deim} without the assumption that $\Vhat_k^T \Vhat_k = I$. 

\begin{lemma}
Let $\widehat{A}_k = \Vhat_k \Shat_k \What_k^T$ be an inexactly computed rank-$k$ truncated SVD of $A$ satisfying the bound \eqref{eq:inexactSVD}, and let $P$ and $Q$ be defined as in \eqref{eq:PQdef}.
Assume that $P^T\Vhat_k$ is invertible, and let $\mathcal{P}=\Vhat_k(P^T\Vhat_k)^{-1}P^T$ be an interpolatory projector. Then 
\begin{equation*}
\norm{A-\mathcal{P}A} \leq \norm{\Vhat_k} \norm{(P^T\Vhat_k)^{-1}} (\E\sigma_{k+1}+\F).
\end{equation*}

\end{lemma}
\begin{proof}
    
 Let $\mathcal{P}=\Vhat_k(P^T\Vhat_k)^{-1}P^T$ be an interpolatory projector. From this we have
\[
\mathcal{P}\Vhat_k = \Vhat_k \rightarrow (I-\mathcal{P})\Vhat_k = 0.
\]
From this,
\begin{align*}
\norm{A-\mathcal{P}A} = \norm{(I-\mathcal{P})A} &= \norm{(I-\mathcal{P})(I-\Vhat_k \Vhat_k^\dagger)A} \\
&\leq \norm{I-\mathcal{P}} \norm{(I-\Vhat_k \Vhat_k^\dagger)A}.
\end{align*}
We can then write, assuming $\mathcal{P}\neq 0$ or $I$,
\[
\norm{I-\mathcal{P}} = \norm{\mathcal{P}} = \norm{\Vhat_k (P^T\Vhat_k)^{-1}P^T} = \norm{\Vhat_k (P^T\Vhat_k)^{-1}} \leq \norm{\Vhat_k}\norm{(P^T \Vhat_k)^{-1}};
\]
note that this differs from the result in \cite[Lemma 4.1]{sorensen2016deim} since we cannot assume $\norm{\Vhat_k}=1$.
Thus altogether we have the bound
\[
\norm{A-\mathcal{P}A} \leq \norm{\Vhat_k} \norm{(P^T\Vhat_k)^{-1}} \norm{(I-\Vhat_k\Vhat_k^\dagger)A}.
\]

We now seek to bound the last term on the right-hand side above. Writing $A = \widehat{A}_k + (A-\widehat{A}_k)$, we have
\begin{equation}
    \norm{(I-\Vhat_k\Vhat_k^\dagger)A} = \norm{(I-\Vhat_k \Vhat_k^\dagger)(A-\widehat{A}_k)} \leq \norm{A-\widehat{A}_k} \leq \E\sigma_{k+1} + \F.
\end{equation}

Substituting this above, we have 
\[
\norm{A-\mathcal{P}A} \leq \norm{\Vhat_k} \norm{(P^T\Vhat_k)^{-1}} (\E\sigma_{k+1}+\F).
\]

\end{proof}

In a completely analogous way, we can consider the right interpolatory projector $\mathcal{Q} = Q(\What_k^TQ)^{-1}\What_k^T$, assuming that $\What^TQ$ is invertible, and can show that 
\begin{equation}
    \Vert A - A\mathcal{Q}\Vert \leq \norm{\What_k}\norm{(\What_k^T Q)^{-1}}(\E\sigma_{k+1}+\F).
\end{equation}

Now, define
\begin{equation}
    \tilde{\eta}_p = \norm{(P^T\Vhat_k)^{-1}}, \qquad 
    \tilde{\eta}_q = \norm{(\What_k^T Q)^{-1}}.
    \label{eq:etas}
\end{equation}
Then 
\begin{align}
    \norm{(I-\mathcal{P})A} & \leq \norm{\Vhat_k}\tilde{\eta}_p(\E\sigma_{k+1}+\F),\\
    \norm{A(I-\mathcal{Q})} &\leq \norm{\What_k}\tilde{\eta}_q (\E\sigma_{k+1}+\F).
\end{align}

The CUR approximation uses $C = AQ$ and $R = P^T A$.
The relevant projection residuals involve the orthogonal projectors
$CC^\dagger$ and $R^\dagger R$.
\begin{lemma}
\label{lem:CR}
Assuming \eqref{eq:inexactSVD} holds, for $C = AQ$ and $R = P^T A$ with $Q$ and $P$ defined in \eqref{eq:PQdef},
\begin{align}
  \norm{(I - CC^\dagger)A} &\leq \norm{\What_k}\,\tilde\eta_q\,(\E\sigma_{k+1}+\F),
  \label{eq:C_res} \\
  \norm{A(I - R^\dagger R)} &\leq \norm{\Vhat_k}\,\tilde\eta_p\,(\E\sigma_{k+1}+\F).
  \label{eq:R_res}
\end{align}
\end{lemma}

\begin{proof}
The proof is analogous to the proof of Lemma 4.2 in \cite{sorensen2016deim}.
\end{proof}

We can then state the result for the error of the CUR decomposition. 
\begin{theorem}
\label{thm:main}
Let $A \in \mathbb{R}^{m\times n}$, $1 \le k < \min(m,n)$.  Suppose the DEIM index
sets $P$ and $Q$ are computed from $\Vhat_k$ and $\What_k$
satisfying \eqref{eq:inexactSVD}, with amplification factors
$\tilde\eta_p$ and $\tilde\eta_q$ defined in \eqref{eq:etas}.  Let $C = AQ$, $R = P^T A$, and
$U = C^\dagger A R^\dagger$.  Then
\begin{equation}
  \norm{A - CUR}
  \leq
  \bigl(\norm{\What_k}\,\tilde\eta_q + \norm{\Vhat_k}\,\tilde\eta_p\bigr)
  \bigl(\E\sigma_{k+1} + \F\bigr).
  \label{eq:main}
\end{equation} 
\end{theorem}
\begin{proof}
    The proof is analogous to the proof of Theorem 4.1 in \cite{sorensen2016deim}.
\end{proof}

\begin{remark}
\label{rem:recovery}
When the SVD is computed exactly, we have $\norm{\Vhat_k} = \norm{\What_k} = 1$,
$\tilde\eta_p = \eta_p$, $\tilde\eta_q = \eta_q$, $\E = 1$, $\F = 0$,
and~\eqref{eq:main} reduces to the result of Sorensen and Embree. 
\end{remark}

We now aim to provide bounds on the quantities $\tilde{\eta}_p$ and $\tilde{\eta}_q$ when $P$ and $Q$ are produced by the (exact) DEIM algorithm run using the inexact SVD factors.     

\begin{lemma}\label{lemma_eta_exact_DEIM}
    For the (exact) DEIM selection algorithm derived above and under the inexact SVD assumptions, 
    \begin{equation}
    \begin{split}
        \tilde{\eta}_p &\leq \sqrt{\frac{mk}{3}}\frac{2^k}{\sigma_{\min}(\Vhat_k)}\\ 
        \tilde{\eta}_q &\leq \sqrt{\frac{nk}{3}}\frac{2^k}{\sigma_{\min}(\What_k)}        
    \end{split}
    \end{equation}
\end{lemma}
\begin{proof}
    Let us prove the result for $\tilde{\eta}_p$; the result for $\tilde{\eta}_q$ follows similarly. Assume $\Vhat_k$ is the approximation of the left singular space coming from the inexact SVD. Then we cannot assume orthonormality. Let $p = DEIM(\Vhat_k)$ denote the row index vector derived from the DEIM selection scheme described above. Let $P = I(:,p)$ so that $P^T\Vhat_k = \Vhat_k(p,:)$. The DEIM selection is precisely the index selection of LU decomposition with partial pivoting, so one can write (as in the proof of Lemma 4.4 from \cite{sorensen2016deim}): 
    \begin{equation*}
        \Vhat_k = LT
    \end{equation*}
    Let $L_1 \equiv L(1:k,1:k)$. Then $\Vhat_k(p,:) = L_1T$ and thus:
    \begin{equation*}
        \tilde{\eta}_p = \norm{(P^T\Vhat_k)^{-1}} = \norm{(L_1T)^{-1}} \leq \norm{T^{-1}}\norm{L_1^{-1}}
    \end{equation*}
    So we seek to bound $\norm{T^{-1}}$ and $\norm{L_1^{-1}}$. To bound $\norm{T^{-1}}$, let $y \in \mathcal{R}^k$ be a unit vector such that $\norm{T^{-1}y} = \norm{T^{-1}}$. Then
    \begin{equation*}
    \begin{split}
        \norm{T^{-1}} &= \norm{T^{-1}y} \\
        &= \norm{\Vhat_k^\dagger \Vhat_k  T^{-1}y}\\
        &\leq \norm{\Vhat_k^\dagger} \norm{\Vhat_k  T^{-1}y}\\
        &\leq \norm{\hat{V}_k^\dagger} \norm{Ly}\\
        &\leq \frac{\norm{Ly}}{\sigma_{\min}(\Vhat_k)}\\
        &\leq \frac{\sqrt{m k}}{\sigma_{\min}(\Vhat_k)}
    \end{split}
    \end{equation*}
    The rest of the proof is analogous to the proof of Lemma 4.4 from \cite{sorensen2016deim} for bounding $\norm{Ly}$ and $\norm{L_1^{-1}}$. 
\end{proof}
\section{DEIM in Finite Precision}
\label{sec:fpDEIM}

Section \ref{sec:inexactSVD} bounds $\tilde\eta_p,\tilde\eta_q$ under the assumption that the
index sets $p,q$ are those produced by the DEIM algorithm applied
exactly to $\Vhat_k,\What_k$. We now ask directly: if DEIM is executed in
floating-point arithmetic and produces an index set $\Phat$ using the matrix
$\Vhat_k$ (and index set $\Qhat$ using the matrix $\What_k$), how do the resulting amplification factors

\begin{equation}
  \widehat\eta_p := \norm{(\Phat^T\Vhat_k)^{-1}}, \qquad \widehat\eta_q := \norm{(\What_k^T\Qhat)^{-1}}
  \label{eq:eta-hat}
\end{equation}
behave as a function of the working precision $u_{D}$? 

Here we assume that the DEIM selection is implemented through Gaussian elimination
with partial pivoting (GEPP) applied to $\Vhat_k$, so when DEIM is executed
in floating-point arithmetic, the pivot sequence can be taken as the pivot sequence of a
finite precision LU factorization of $\Vhat_k$. Similarly as before, let $\Phat= I(:,\hat{p})$ so that $\Phat^T\Vhat_k = \Vhat_k(\hat{p},:)$. Standard backward error analysis
for GEPP \cite{higham2002accuracy} gives computed factors $\Lhat,\That$ satisfying
\begin{equation}
  \Vhat_k + \Delta V = \Lhat\,\That, \qquad \norm{\Delta V} \le \gamma^D_k\,\rho_k\,\norm{\Vhat_k},
  \label{eq:finite-lu-backward}
\end{equation}
where 
 $\rho_k = \|\,  |\Lhat| |\That|\,\|/\|\Vhat_k \|$ , and where the row permutation recording the pivot choices
made by $\Lhat$ is precisely the finite precision DEIM index set $\Phat$. We denote $\Phat^T \Delta V = \Delta V(\hat{p},:) \equiv \Delta V_1$.

\begin{lemma}
\label{lem:eta-bound-fp}
Let $\hat p=\mathrm{DEIM}(\Vhat_k)$ be the index set produced when the DEIM index selection is implemented through GEPP applied to $\Vhat_k$ in floating point arithmetic with unit roundoff $u_{D}$, with selection matrix $\Phat$ and  
$\rho_k$ defined above. If 
\begin{equation}
    ku_D < 1 \quad\text{and}\quad \left(1+ 2^k \sqrt{\frac{mk}{3}} \right)\gamma^D_k\rho_k\|\Vhat_k\| < \sigma_{min}(\Vhat_k),
    \label{eq:assumptions}
\end{equation}
then $\Phat^T \Vhat_k$ is nonsingular and
\begin{equation}
  \hat\eta_p \leq \frac{2^k \sqrt{\frac{mk}{3}}}{\sigma_{min}(\Vhat_k)- \left(1+2^k\sqrt{\frac{mk}{3}} \right) \gamma^D_k \rho_k \|\Vhat_k\|}.
  \label{eq:eta-bound-fp}
\end{equation}
The analogous bound holds for $\hat\eta_q$ with $\What_k,n$ in place of $\Vhat_k,m$, respectively.
\end{lemma}

\begin{proof}

From \eqref{eq:finite-lu-backward}, we have
\[
\Phat^T\Vhat_k = \Lhat{L}_1 \That - \Delta V_1, \quad \|\Delta V_1 \| \leq \gamma^D_k \rho_k \|\Vhat_k\|.
\]
By Weyl's inequality and the assumption in \eqref{eq:assumptions}, 
\[
\sigma_{min}(\Vhat_k + \Delta V) \geq \sigma_{min}(\Vhat_k)-\gamma^D_k \rho_k \|\Vhat_k\| > 0.
\]
Then $\Vhat_k + \Delta V$ has full column rank and $\That$ is nonsingular. 

Since $\That^{-1} = (\Vhat_k + \Delta V)^\dagger \Lhat$, the bound $\|\Lhat\|\leq \sqrt{mk}$ gives
\[
\| \That^{-1}\|\leq \frac{\sqrt{mk}}{\sigma_{min}(\Vhat_k-\gamma^D_k \rho_k \|\Vhat_k \|)}.
\]
Combining this with $\|\Lhat_1^{-1}\|\leq 2^k/\sqrt{3}$ gives
\[
\| (\Lhat_1 \That)^{-1}\| \leq \frac{2^k \sqrt{\frac{mk}{3}}}{\sigma_{min}(\Vhat_k)-\gamma^D_k \rho_k \|\Vhat_k\|}.
\]
Applying Weyl's inequality again gives
\begin{align*}
\sigma_{min}(\Phat^T\Vhat_k) &\geq \sigma_{min}(\Lhat_1 \That) - \|\Delta V_1 \| \\
&\geq \frac{\sigma_{min}(\Vhat_k)-\gamma^D_k \rho_k \|\Vhat_k\|}{2^k \sqrt{\frac{mk}{3}}}-\gamma^D_k\rho_k\|\Vhat_k\|\\
&= \frac{\sigma_{min}(\Vhat_k) - \left(1+2^k\sqrt{\frac{mk}{3}}\right)\gamma^D_k \rho_k \|\Vhat_k\|}{2^k \sqrt{\frac{mk}{3}}},
\end{align*}
which is positive by \eqref{eq:assumptions} and which gives the bound \eqref{eq:eta-bound-fp}.

\end{proof}

\begin{remark}
We can use \eqref{eq:eta-bound-fp} to give guidance on how to choose $u_{D}$. 
We can say that the finite precision bound remains close to the exact arithmetic bound whenever
\[
\left(1+ 2^k\sqrt{\frac{mk}{3}} \right) \frac{\gamma^D_k \rho_k \|\Vhat_k\|}{\sigma_{min}(\Vhat_k)} \ll 1.
\]
Note that here we can only say that the two worst case bounds will be of similar magnitude; we can not necessarily say that the finite precision algorithm produces an amplification factor close to that obtained by the exact algorithm. Assuming $ku_D \ll 1$, we have $\gamma^D_k \approx ku_D$. Unless we have computed a very bad inexact SVD, we still expect $\Vert \Vhat_k\Vert$ and $\sigma_{\min}(\Vhat_k)$ to be close to 1 (we will make this more concrete in the following section). Then $\|\Vhat_k\|/\sigma_{min}(\Vhat_k)\approx 1$, and the above condition becomes 
\begin{equation}
k u_{D} \rho_k \left( 1+ 2^k \sqrt{\frac{mk}{3}} \right) \ll 1.
\label{eq:uDEIM}
\end{equation}
Unless we have computed a very bad inexact SVD, we still expect $\Vert \Vhat_k\Vert$ and $\sigma_{\min}(\Vhat_k)$ to be close to 1 (we will make this more concrete in the following section). This means that we only need $u_{D} < 1/(k \rho_k)$. While we can not know $\rho_k$ ahead of time, it is expected that in most practical cases, this factor will be moderate. This means that we can get away with using potentially \emph{very} low precision in the DEIM selection algorithm.

If we also want to guarantee that the amplification factors $\widehat{\eta}_p$ do not grow significantly, we may require that, e.g., 
\[
\gamma^D_k \rho_k \|\Vhat_k\| \sqrt{\frac{mk}{3}}\;\frac{2^k}{\sigma_{\min}(\Vhat_k) - \gamma^D_k\rho_k\norm{\Vhat_k}} \leq 1.
\]
On the left-hand side above, we have approximately $k \rho_k u_{D}$ multiplied by (approximately) the amplification factor for exact DEIM. This says that if we expect a very large amplification factor due to the DEIM algorithm itself, then we must use higher precision. Unfortunately, we can not know the amplification factor ahead of time. Experimentally, we have found that the amplification factors are usually moderate, which echoes the insight of \cite{sorensen2016deim} where the authors comment that they never observed exponential growth over very extensive testing. This suggests that low precision is in many cases suitable for DEIM. 

\end{remark}
\section{Application to Specific SVD Algorithms}
\label{sec:specSVD}

We now discuss two concrete, inexact SVD algorithms that lead to different $(\theta,\phi)$ pairs. The first is the \texttt{GESVD} routine in LAPACK \cite{anderson1999lapack}, and the second is the randomized SVD \cite{halko2011finding}, which has been analyzed in finite precision in \cite{connolly2022randomized}.

\subsection{LAPACK's \texttt{GESVD}}
\hspace{5mm}For \texttt{GESVD} in LAPACK \cite{anderson1999lapack}, we have that the computed result satisfies the backward error result
\[
A+E = (\Vhat+\Delta V) \Shat (\What+\Delta W)^T \equiv \widetilde{V}\Shat \widetilde{W}^T,
\]
with $\Vert E\Vert \leq p(m,n)u_S\Vert A\Vert$, where $\widetilde{V} = \Vhat + \Delta V$ and $\widetilde{W}=\What + \Delta W$ are orthogonal, with $\Vert \Delta V\Vert, \Vert \Delta W \Vert \leq p(m,n)u_S$, and where $p(m,n)$ is a modestly growing function of $m$ and $n$. The right-hand side of the above expression gives the exact SVD for $A+E$. Note that here $\Delta V$ has a different definition than the previous section. 

Let $\widehat{A}_k = \Vhat_k \Shat_k \What_k^T$ and $\widetilde{A}_k = \widetilde{V}_k \Shat_k \widetilde{W}_k^T = (A+E)_k$.
Writing
\[
A-\Vhat_k\Shat_k\What_k^T = A - \widehat{A}_k = (A - (A+E)) + ((A+E)-\widetilde{A}_k) + (\widetilde{A}_k - \widehat{A}_k),
\]
we have
\begin{align*}
\Vert A- \Vhat_k\Shat_k\What_k^T\Vert &\leq \Vert E \Vert +  \hat\sigma_{k+1} + \|\widetilde{A}_k - \widehat{A}_k\| \\
&\leq \sigma_{k+1} + 2p(m,n)u_S\Vert A \Vert + \|\widetilde{A}_k - \widehat{A}_k\| . 
\end{align*}
Ignoring higher order terms in $u_S$, we can bound 
\[
\|\widetilde{A}_k - \widehat{A}_k \| \leq \|\Delta V_k \Shat_k \widetilde{W}_k^T\| + \|\widetilde{V}_k \Shat_k \Delta W_k^T\| \leq 2p(m,n)u_S \|A\|.
\]
Thus together, we have
\[
\Vert A- \Vhat_k\Shat_k\What_k^T\Vert \leq \sigma_{k+1} + 4p(m,n)u_S\|A\|,
\]
and thus for LAPACK's \texttt{GESVD} function, we have
\begin{equation}
    \theta = 1, \qquad \phi = 4p(m,n)u_S\Vert A \Vert.
    \label{eq:thetaphigesvd}
\end{equation}

Further, using the results in \cite{anderson1999lapack}, we know that (to first order),
\begin{equation}
\sigma_{min}(\Vhat_k) \geq 1 - \frac{1}{2}p(m,n)u_S+O(u_S)^2, \qquad \Vert \Vhat_k\Vert \leq 1 + \frac{1}{2}p(m,n)u_S+O(u_S)^2,
\label{eq:Vbounds}
\end{equation}
and analogously for $\What_k$.

Combining \eqref{eq:main} with \eqref{eq:thetaphigesvd} and \eqref{eq:Vbounds}, we can say that if we choose
\begin{equation}
u_S \ll \frac{\sigma_{k+1}}{4p(m,n)\sigma_1} = \frac{\Vert A-A_k\Vert}{4p(m,n)\Vert A\Vert},
\label{eq:usvdcond1}
\end{equation}
then it is likely that we will not see significant effects of using low precision for the SVD computation, assuming also that $(1/2)p(m,n)u_S\ll 1$. Considering the right-hand side above, it is clear that the coarser the desired approximation (the smaller the $k$, or the closer $\sigma_{k+1}$ is to $\sigma_1$), the lower precision one can use without detriment.

\subsection{Randomized SVD}

\hspace{5mm}We now derive the constants $\theta$ and $\phi$ for the randomized SVD algorithm of
Connolly, Higham, and Pranesh~\cite{connolly2022randomized} (Algorithm 2.2, built on the rangefinder
of Halko, Martinsson, and Tropp~\cite{halko2011finding}).

By Theorem~2.5 of \cite{connolly2022randomized}, if we execute Algorithm 2.2 in \cite{connolly2022randomized} in precision $u^S$ except with the small SVD on line 2 of Algorithm 1.2 computed exactly, we have
\begin{equation}
\|A-\Vhat_k \Shat_k \What_k^{T}\|_F
\le \|A-QQ^{T}A\|_F
+\Big((1+k^{1/2})\widetilde{\gamma}^S_{mk}+k^{1/2}\gamma^S_n(1+\widetilde{\gamma}^S_{mn})+k^{1/2}(\widetilde{\gamma}^S_{mn})^2\Big)\|A\|_F, 
\label{eq:errrsvd}
\end{equation}
where $Q\in\mathbb{R}^{m\times k}$.

For a Gaussian test matrix $\Omega \in \mathbb{R}^{n \times (k+s)}$ with oversampling
parameter $s$, and free parameters $\tau, \alpha>1$, Theorem~10.7 of \cite{halko2011finding}
(as invoked by \cite{connolly2022randomized}) gives
\begin{equation}
\|A-QQ^{T}A\|_F
\le \left(1+\tau\sqrt{\frac{3k}{s+1}}\right)\left(\sum_{j>k}\sigma_j^2\right)^{1/2}
+ \alpha\tau\, \frac{e\sqrt{k+s}}{s+1}\,\sigma_{k+1}, 
\label{eq:AQQTA}
\end{equation}
with probability at least $1-\big(2\tau^{-s}+e^{-\alpha^2/2}\big)$.

Substituting \eqref{eq:AQQTA} into \eqref{eq:errrsvd},
\begin{align*}
\|A-\Vhat_k \Shat_k \What_k^{T}\|
&\le \|A-\Vhat_k \Shat_k \What_k^{T}\|_F\\
&\le \alpha\tau\,\frac{e\sqrt{k+s}}{s+1}\,\sigma_{k+1}\\
&\phantom{\leq}+\left(1+\tau\sqrt{\frac{3k}{s+1}}\right)\left(\sum_{j>k}\sigma_j^2\right)^{1/2}
\\
&\phantom{\leq}+\Big((1+k^{1/2})\widetilde{\gamma}^S_{mk}+k^{1/2}\gamma^S_n(1+\widetilde{\gamma}^S_{mn})+k^{1/2}(\widetilde{\gamma}^S_{mn})^2\Big)\|A\|_F.
\end{align*} 
We therefore have
\begin{equation*}
\begin{split}
    \theta &= \alpha\tau\,\frac{e\sqrt{k+s}}{s+1}, \\
    \phi  &= \left(1+\tau\sqrt{\frac{3k}{s+1}}\right)\left(\sum_{j>k}\sigma_j^2\right)^{1/2} + \Big((1+k^{1/2})\widetilde{\gamma}^S_{mk}+k^{1/2}\gamma^S_n(1+\widetilde{\gamma}^S_{mn})+k^{1/2}(\widetilde{\gamma}^S_{mn})^2\Big)\|A\|_F.
\end{split}    
\end{equation*}

First, we can notice that the $\theta$ quantity does not depend on the finite precision error, but only on the error from the randomized low rank approximation. While the $\phi$ term involves more parameters and the bounds can only be stated in a probabilistic sense, we still have the same general conclusion as in the LAPACK case (although now stated in a different norm). Ignoring dimensional constants for the sake of reasoning, we can say from the definition of $\phi$ above that the rounding error is likely to remain negligible as long as we choose 
\begin{equation}
u_S \ll \Vert A-A_k\Vert_F/\Vert A\Vert_F,
    \label{eq:usvdcond2}
\end{equation}
assuming that we have chosen our parameters so that $\tau \sqrt{3k/(s+1)}$ remains small. 
In other words, again we have that the coarser the desired low rank approximation, the lower the precision we can use. 

In both the randomized and the deterministic SVD cases, if the user has some estimate of the relative approximation error they anticipate, this can be used to select the precision used in the SVD computation.

We note that, as stated above, the analysis in \cite{connolly2022randomized} assumes that the SVD of the small matrix is computed exactly during the randomized SVD algorithm; for our purposes we can assume that higher (say, double) precision could be used for this step as it is relatively inexpensive. Further, the (stable) Householder QR algorithm is used for computing the matrix $Q$ from the rangefinder step. We can therefore assume that the computed $\Vhat_k$ and $\What_k$ satisfy bounds similar to those in \eqref{eq:Vbounds}.

\section{Numerical Experiments}
\label{sec:numexp}

In this section we reconstruct Examples 1 and 3 from \cite{sorensen2016deim}. We use different algorithms and precisions to compute the SVD, and in all scenarios, we run DEIM in fp64, fp32, fp16 and q52 (quarter precision with 5 exponent bits and 2 explicit significand bits); see Table \ref{tab:precs} for properties of these precisions. We emulate fp16 and q52  arithmetic using the chop function for MATLAB \cite{higham2019simulating} and we always compute the core matrix $U$ as $\texttt{C}\backslash \texttt{A} / \texttt{R}$ in fp64. We plot the error $\Vert A-\widehat{C}U\widehat{R}\Vert_2$ for each DEIM precision and the best low-rank approximation error given by $\sigma_{k+1}$ for comparison versus the rank $k$. For the same settings we also plot the values of $\widehat{\eta}_p$ versus $k$. MATLAB code for reproducing all plots in this section can be found at \url{https://github.com/thanasisGiannis/mp_DEIM_CUR}. \par 

  \begin{table}
    \centering
    \begin{tabular}{lccll}
        \toprule
        & Signif. & Exp. & Range  ($f_\mathrm{max}/f_\mathrm{min}$)    & Unit roundoff $u$ \\
        & bits &  bits   \\
        \midrule
        fp64          & $52$  & $11$ & $2^{2046}\approx10^{616}$  & $2^{-53} \approx1\times10^{-16}$ \\
        fp32         & $23$  & $8$  & $2^{254}\approx10^{76}$   & $2^{-24} \approx6\times10^{-8}$ \\
        fp16        & $10$  & $5$  & $2^{30}\approx10^{9}$    & $2^{-11} \approx5\times10^{-4}$ \\
        fp8 (q52)  & $2$   & $5$  & $2^{30}\approx10^{9}$    & $2^{-3}  \approx1\times10^{-1}$ \\
        \bottomrule
    \end{tabular}
   \caption{Properties of precisions used in experiments. \label{tab:precs}}
    \end{table}

\subsection{Sparse Non-Negative Random Matrix}\label{section:exp1}
\hspace{5mm}We slightly modify a test case (Example 1) from \cite{sorensen2016deim}. Here we construct a sparse, nonnegative matrix $A$ of size $3000 \times 300$ of the form
\[
A = \sum_{j=1}^{10}\frac{2}{j}x_j y_j^T + \sum_{j=11}^{300} \frac{1}{j} x_j y_j^T,
\]
where $x_j \in\mathbb{R}^{3000}$ and $y_j\in\mathbb{R}^{300}$ are sparse vectors with random nonnegative entries computed via the MATLAB command $x_j = \texttt{sprand}(3000,1,0.025)$ and $y_j = \texttt{sprand}(300,1,0.025)$.

We test four different scenarios for the input SVD to DEIM: (1) LAPACK's SVD routine in fp64, (2) LAPACK's SVD routine in fp32, (3) a randomized SVD with oversampling parameter $p=10$ in fp64, and (4) a randomized SVD with oversampling parameter $p=10$ in fp32. For each scenario, we run DEIM in fp64, fp32, fp16, and q52 (quarter precision with 5 exponent bits and 2 explicit significand bits), for $k=1:50$. In Figure \ref{fig:ex1deim}, we can see that there is no discernible difference between the results produced by DEIM in fp64, fp32, and fp16. The error when $u_D =$ q52 varies only very slightly. There is further very little difference between the behavior of the different SVD computation scenarios. As predicted by the theory, we still expect to end up with an almost orthogonal matrix for $\Vhat_k$, even when computed in lower precision and/or using a randomized SVD. 

In Figure \ref{fig:ex1eta} we plot the value of $\widehat{\eta}_p$. Again we can see that there is no discernible difference between the values of $\widehat{\eta}_p$ for $u_D=$ fp64, fp32, or fp16. It is clear that the values of $\widehat{\eta}_p$ differ when $u_D=$ q52, but the difference is at most around a factor of 2. Comparing across plots, we see that there is a slight difference between the results produced by LAPACK's SVD and the randomized SVD, but in both cases, there is no noticeable difference between fp64 and fp32.

\begin{figure}[h]
    \centering
    \includegraphics[trim=1.5cm 6cm 1.5cm 6cm, clip, width=0.45\linewidth]{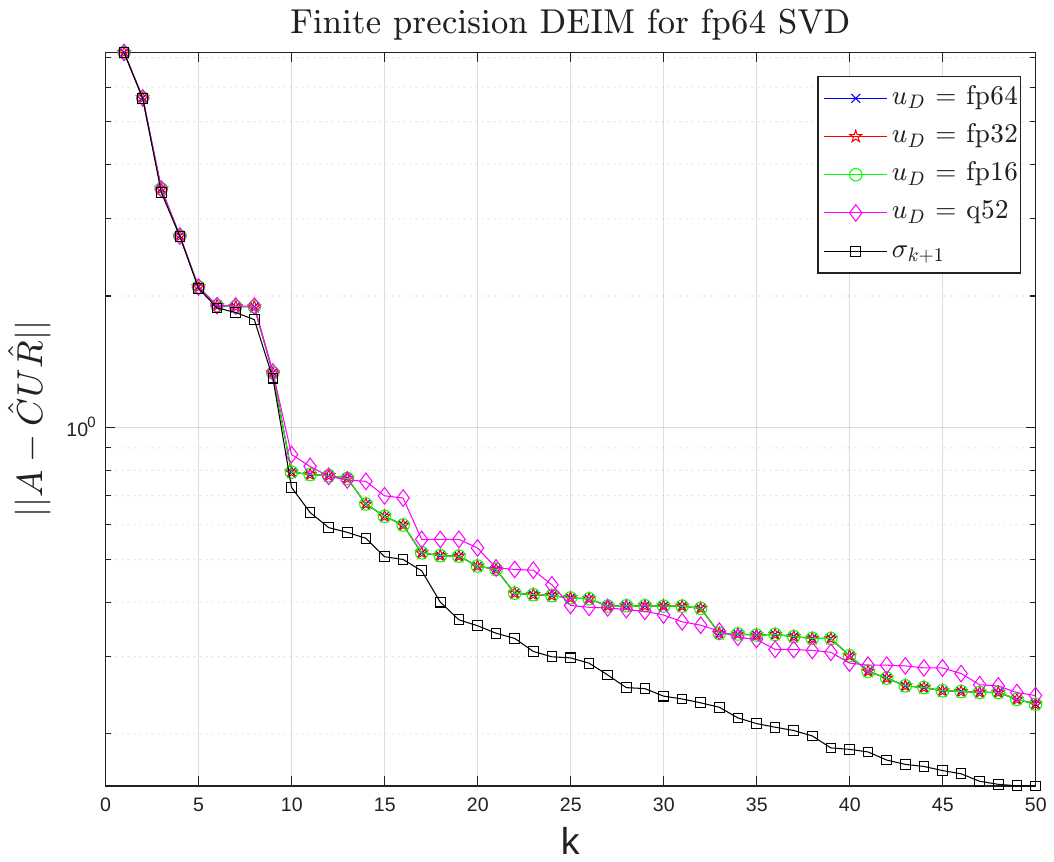} 
    \includegraphics[trim=1.5cm 6cm 1.5cm 6cm, clip, width=0.45\linewidth]{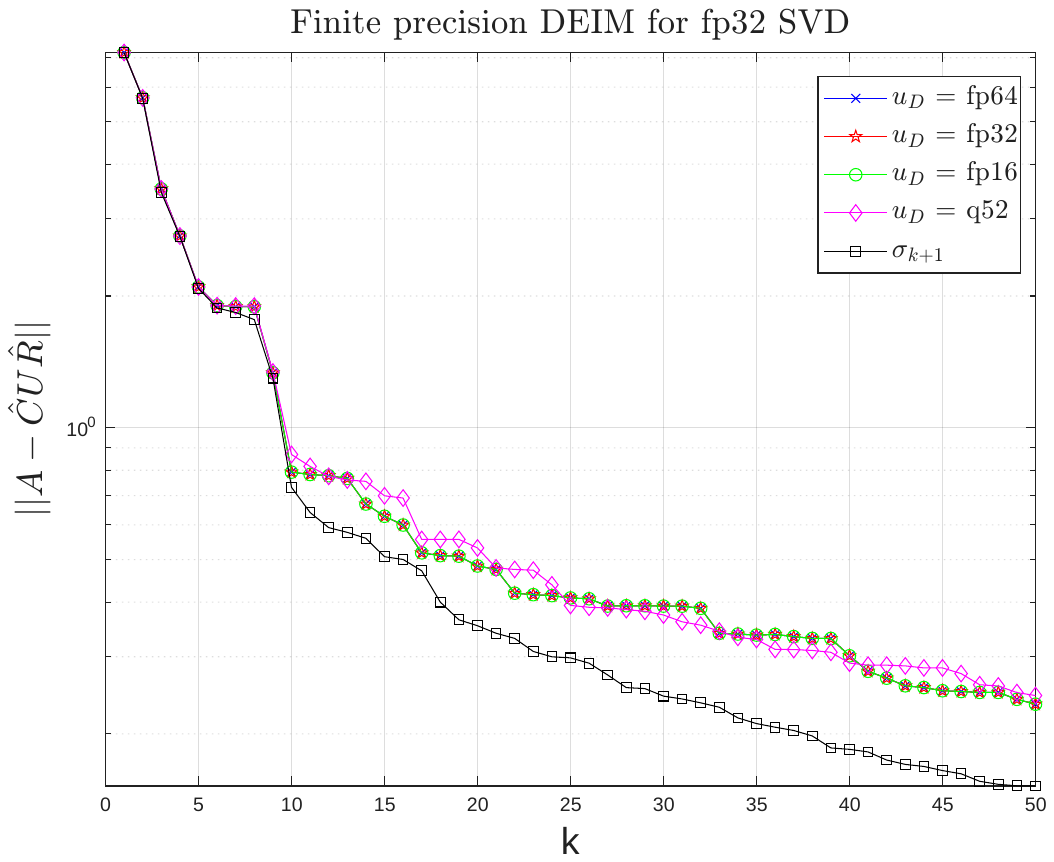} \\
    \includegraphics[trim=1.5cm 6cm 1.5cm 6cm, clip, width=0.45\linewidth]{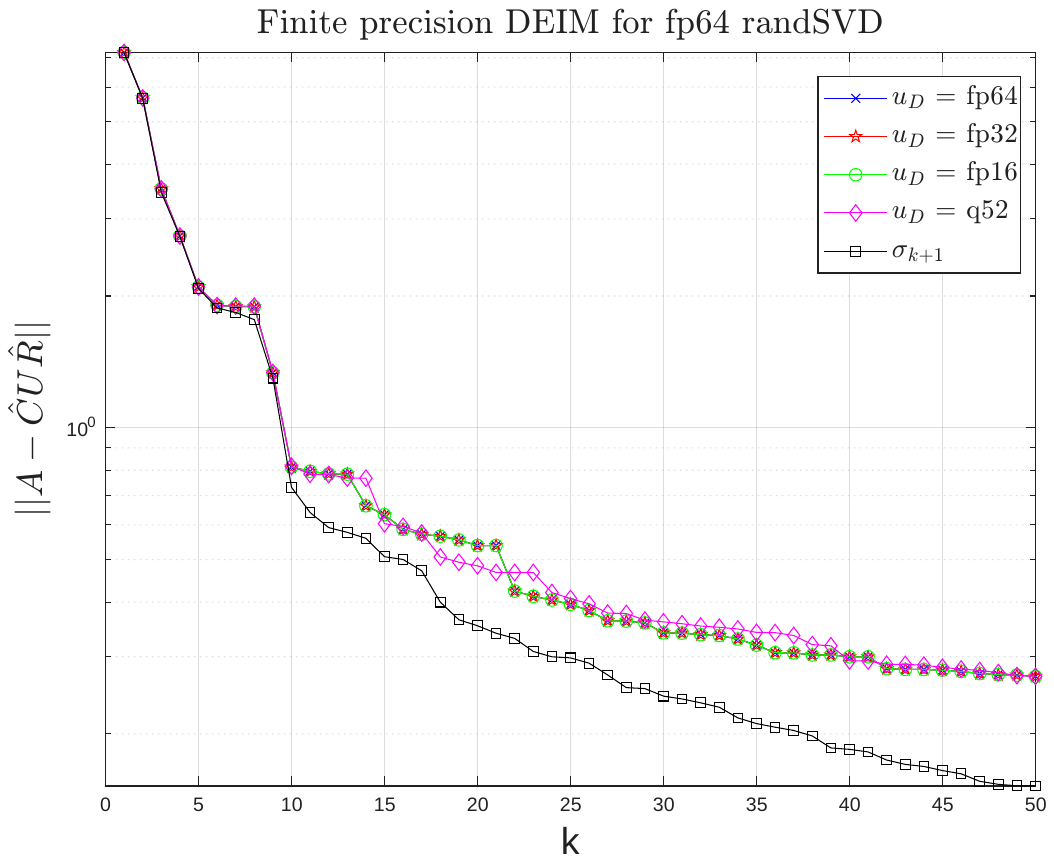} 
    \includegraphics[trim=1.5cm 6cm 1.5cm 6cm, clip, width=0.45\linewidth]{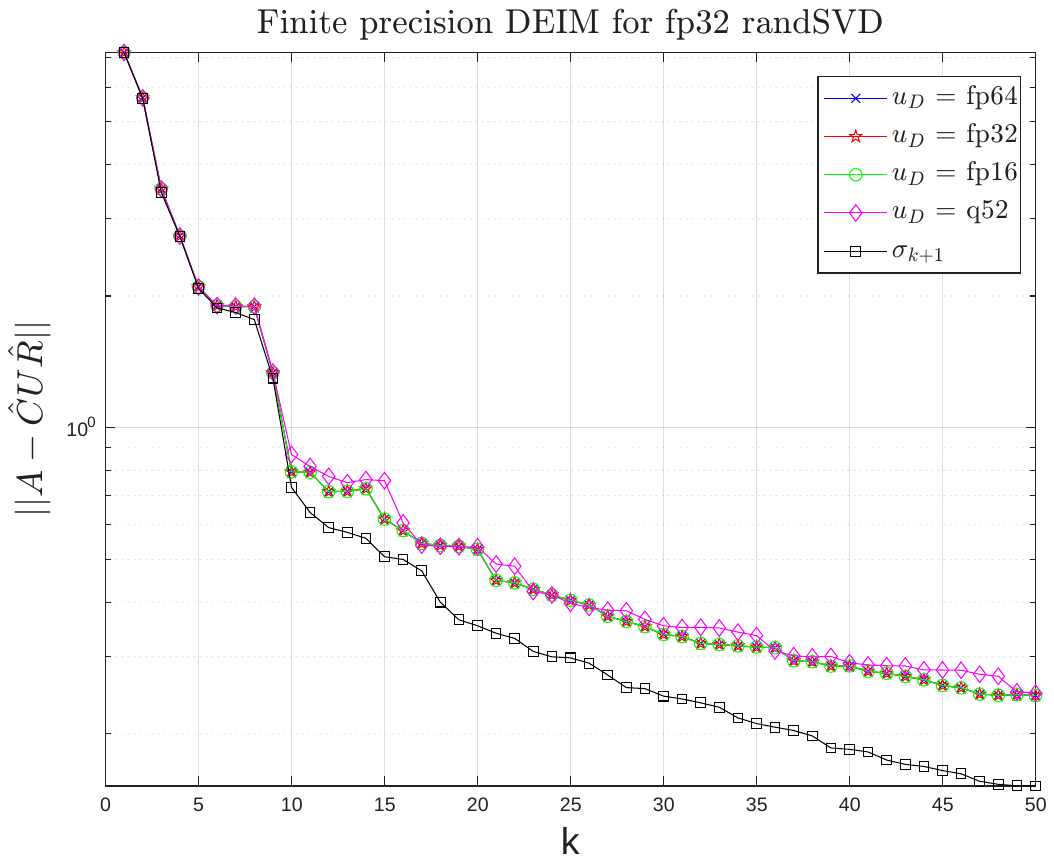} 
    \caption{Error $\Vert A-\widehat{C}U\widehat{R}\Vert_2$ for DEIM in various precisions, where the input SVD is computed using the LAPACK routine (top row) and a randomized SVD with oversampling $p=10$ (bottom row), both computed in fp64 (left column) and fp32 (right column). }
    \label{fig:ex1deim}
\end{figure}

\begin{figure}
    \centering
    \includegraphics[trim=1.5cm 6cm 1.5cm 6cm, clip,width=0.45\linewidth]{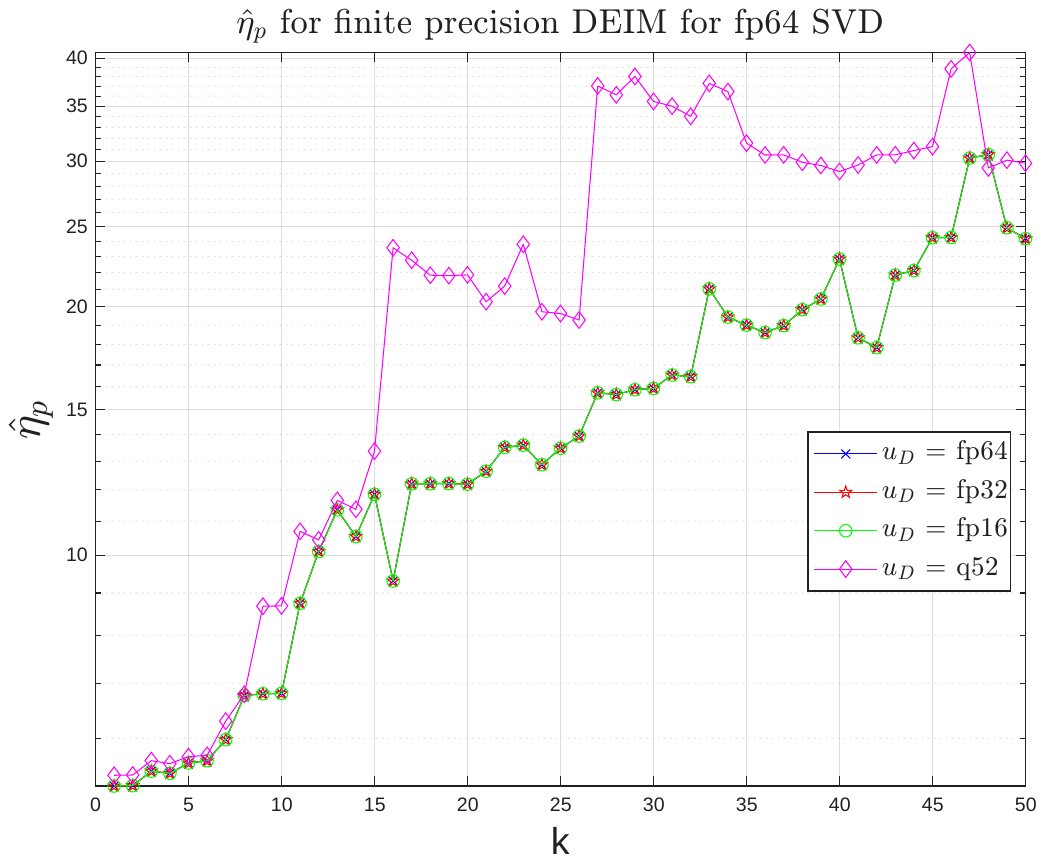} 
    \includegraphics[trim=1.5cm 6cm 1.5cm 6cm, clip,width=0.45\linewidth]{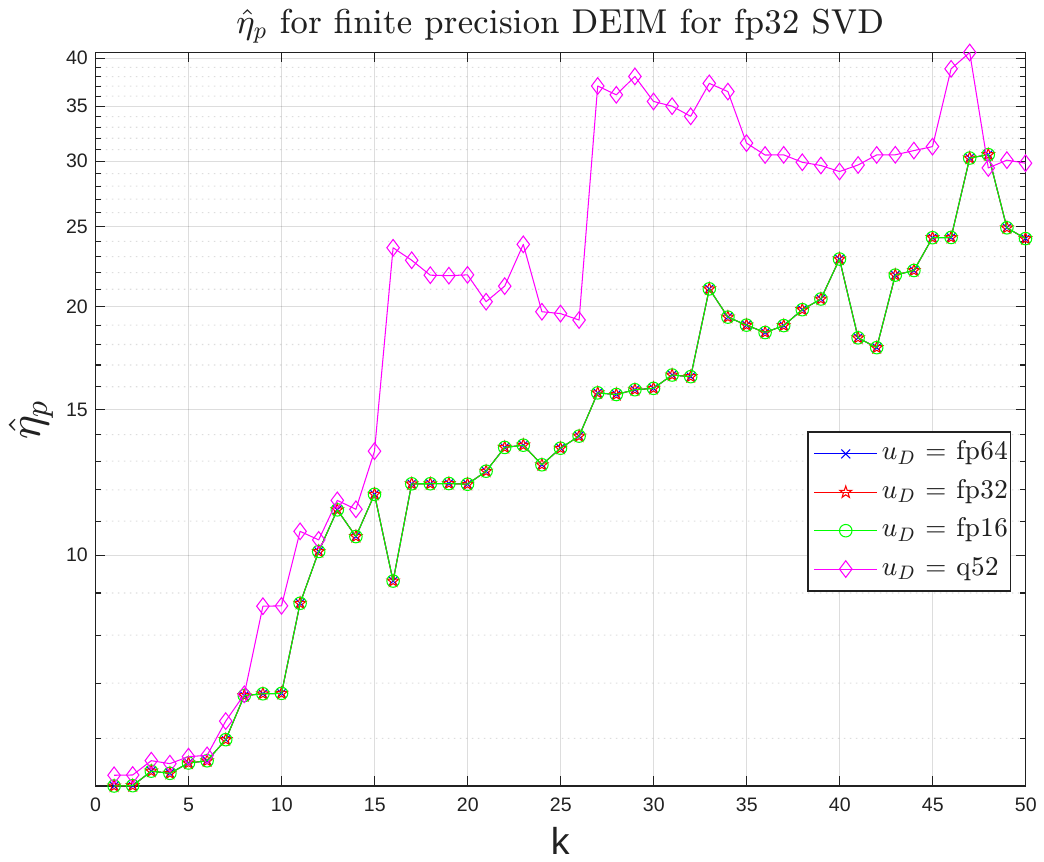} \\
    \includegraphics[trim=1.5cm 6cm 1.5cm 6cm, clip,width=0.45\linewidth]{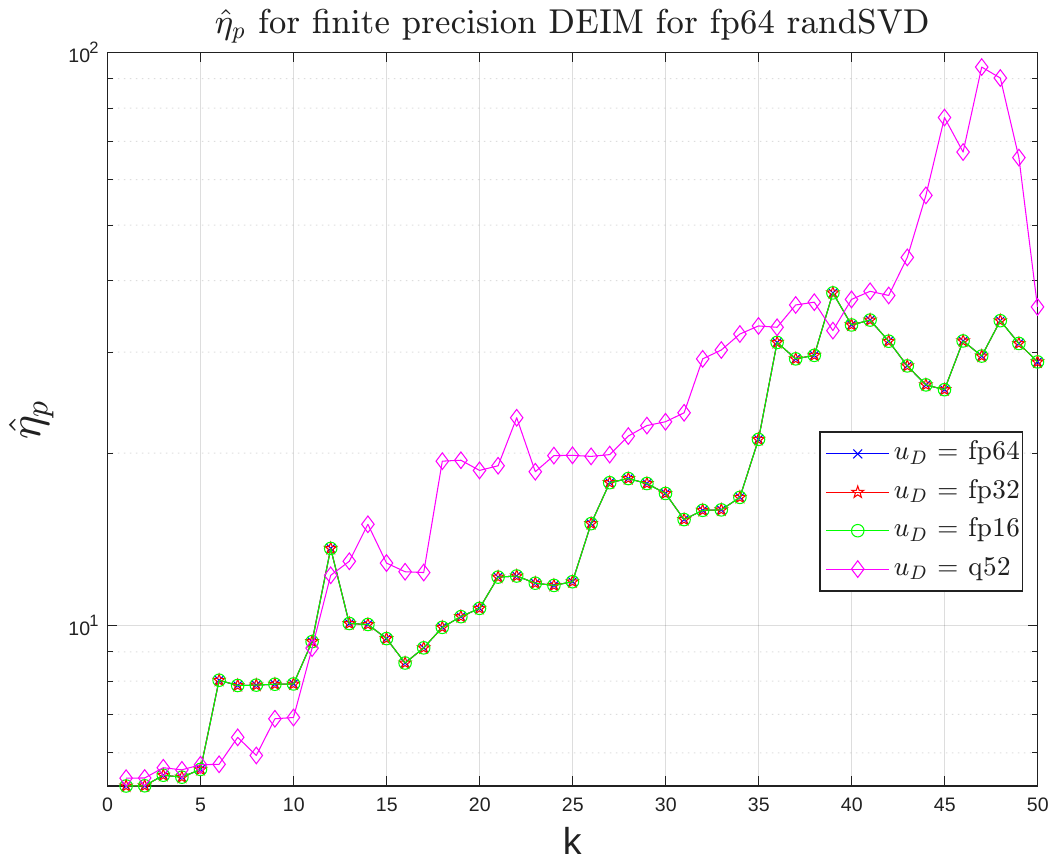} 
    \includegraphics[trim=1.5cm 6cm 1.5cm 6cm, clip,width=0.45\linewidth]{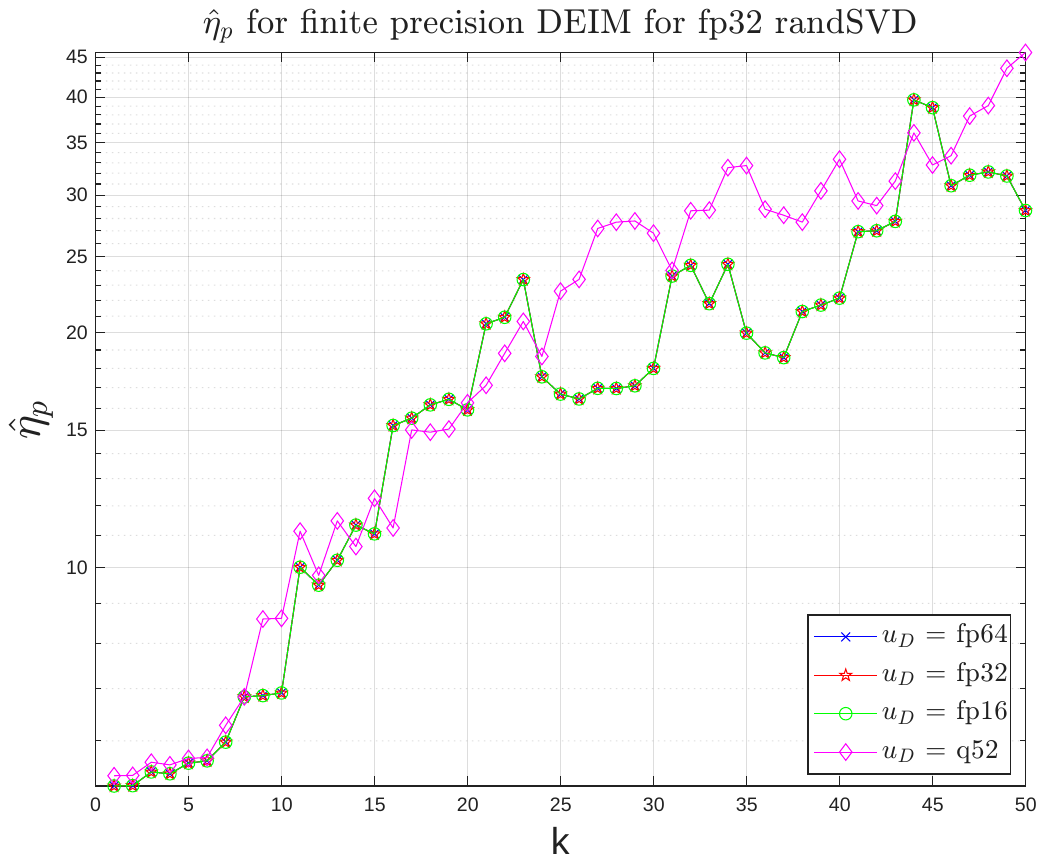} 
    \caption{Value of the amplification factor $\widehat{\eta}_p$ for DEIM in various precisions, where the input SVD is computed using the LAPACK routine (top row) and a randomized SVD with oversampling $p=10$ (bottom row), both computed in fp64 (left column) and fp32 (right column). }
    \label{fig:ex1eta}
\end{figure}

\subsection{Cancer Genetics Data Set}\label{section:exp2}

\hspace{5mm}In this test case we investigate Example 3 from \cite{sorensen2016deim}. The matrix A contains data of 22,283 probes applied to 107 patients, taken from the GSE10072 cancer genetics data set from the National Institutes of Health. Each $(j,k)$ matrix entry reflects how strongly patient k responds to probe $j$. As in \cite{sorensen2016deim} we center the data by subtracting the mean of each row from all the entries in that row. \par 
As in the previous experiment, we test four different scenarios for the input SVD to DEIM: (1) LAPACK's SVD routine in fp64, (2) LAPACK's SVD routine in fp32, (3) a randomized SVD with oversampling parameter $p=10$ in fp64, and (4) a randomized SVD with oversampling parameter $p=10$ in fp32. For each scenario, we run DEIM in fp64, fp32, fp16, and q52 (quarter precision with 5 exponent bits and 2 explicit significand bits), for $k = 1:50$. In Figure \ref{fig:ex2deim}, we can see that there is no significant difference in the error between the different precisions in DEIM. The error when $u_D = q52$ shows only a modest additional variation. However, the choice of the SVD method has a much larger effect than the DEIM working precision. \par
In Figure \ref{fig:ex2eta}, for the same scenarios, we plot the value of $\widehat{\eta}_p$. Although the working precision can change the selected indices, the amplification factors remain of the same order.\par 
The experiments indicate that the CUR approximation is less sensitive to the DEIM working precision, down to fp16, and more strongly affected by the accuracy of the SVD approximation. \par

\begin{figure}
    \centering
    \includegraphics[trim=1.5cm 6cm 1.5cm 6cm, clip,width=0.45\linewidth]{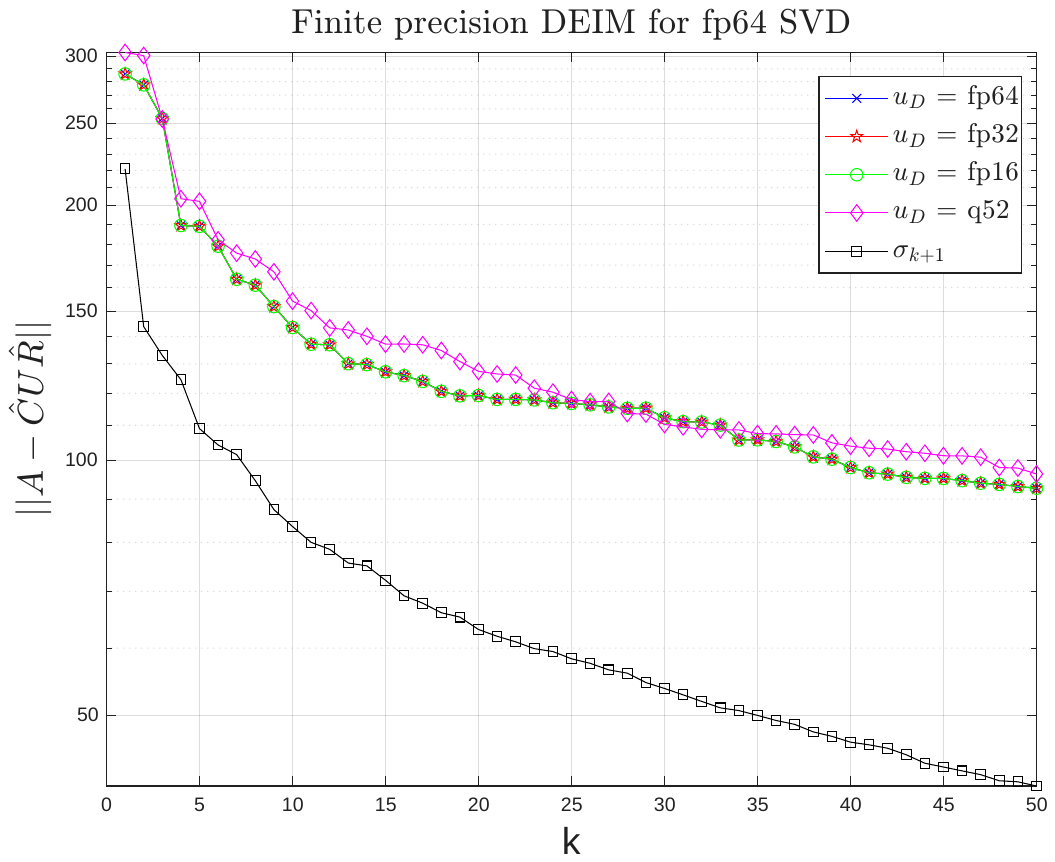} 
    \includegraphics[trim=1.5cm 6cm 1.5cm 6cm, clip,width=0.45\linewidth]{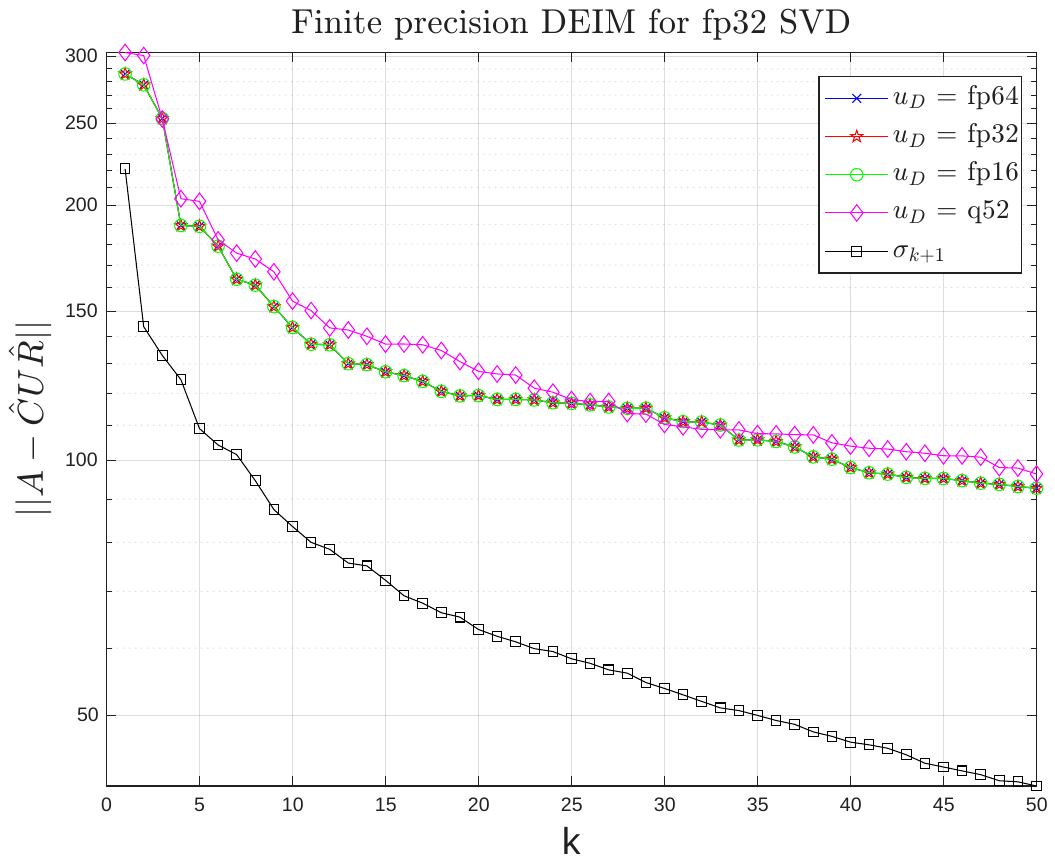} \\
    \includegraphics[trim=1.5cm 6cm 1.5cm 6cm, clip,width=0.45\linewidth]{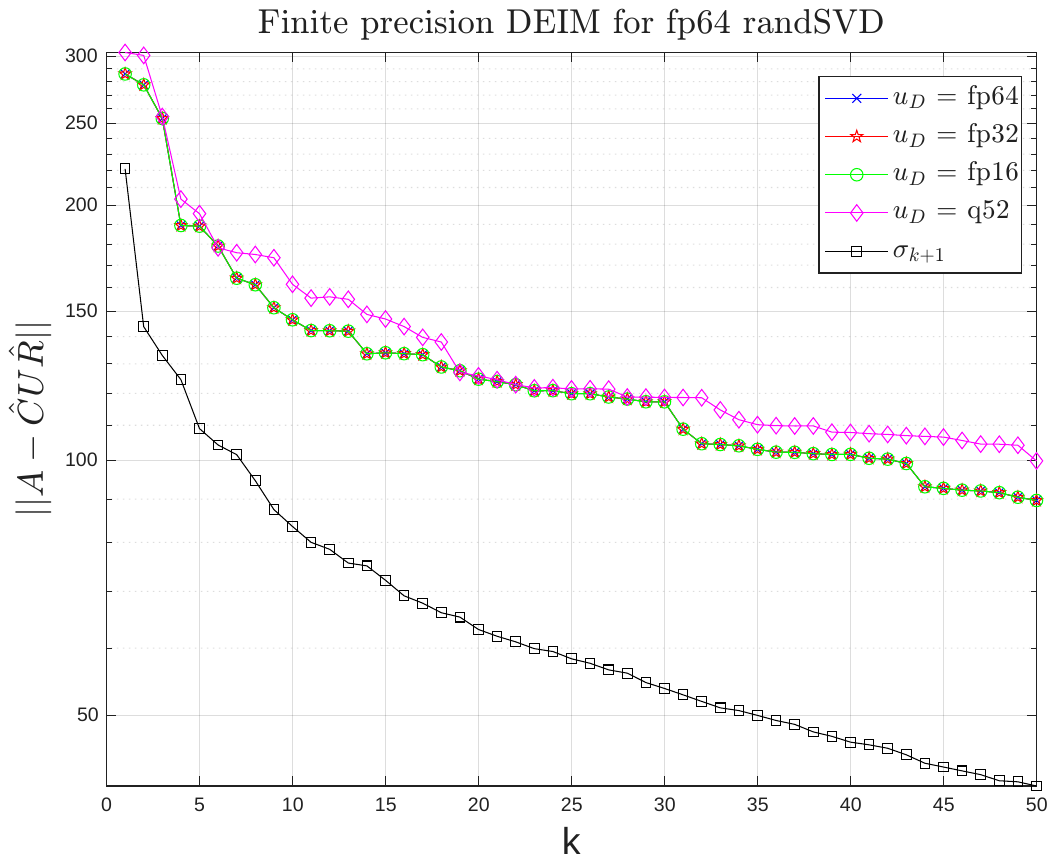} 
    \includegraphics[trim=1.5cm 6cm 1.5cm 6cm, clip,width=0.45\linewidth]{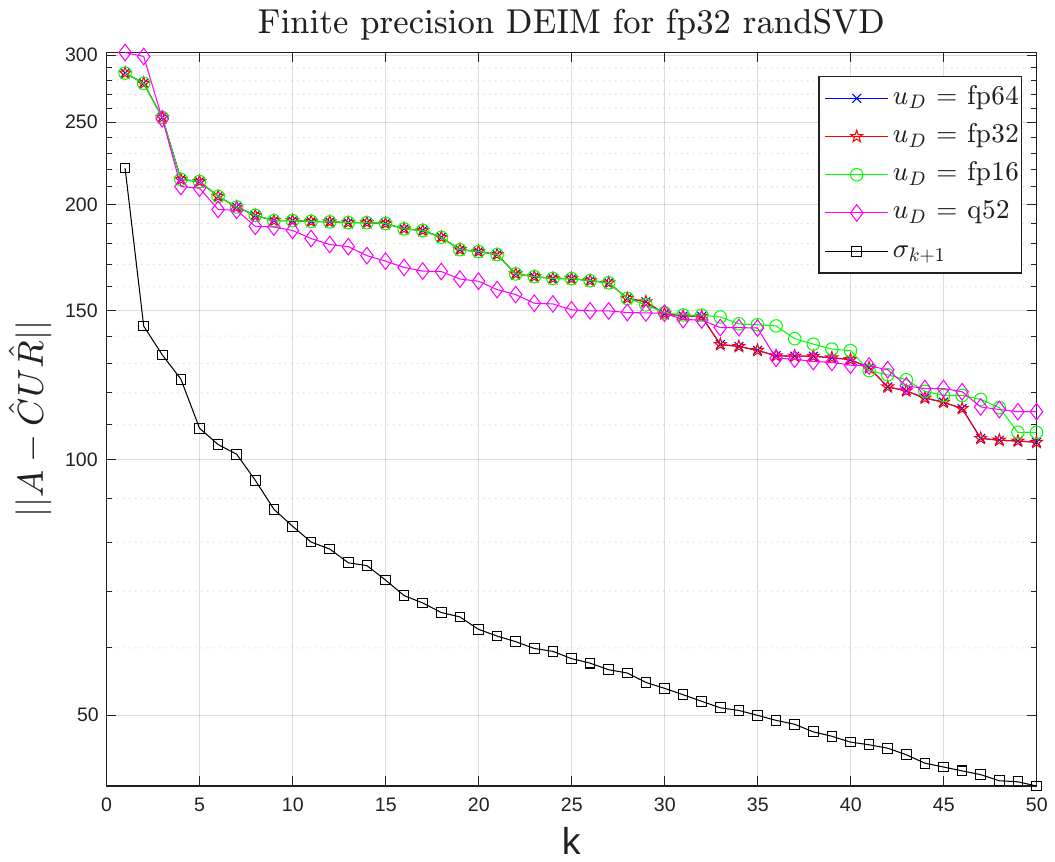} 
    \caption{Error $\Vert A-\widehat{C}U\widehat{R}\Vert_2$ for DEIM in various precisions, where the input SVD is computed using the LAPACK routine (top row) and a randomized SVD with oversampling $p=10$ (bottom row), both computed in fp64 (left column) and fp32 (right column). }
    \label{fig:ex2deim}
\end{figure}

\begin{figure}
    \centering
    \includegraphics[trim=1.5cm 6cm 1.5cm 6cm, clip,width=0.45\linewidth]{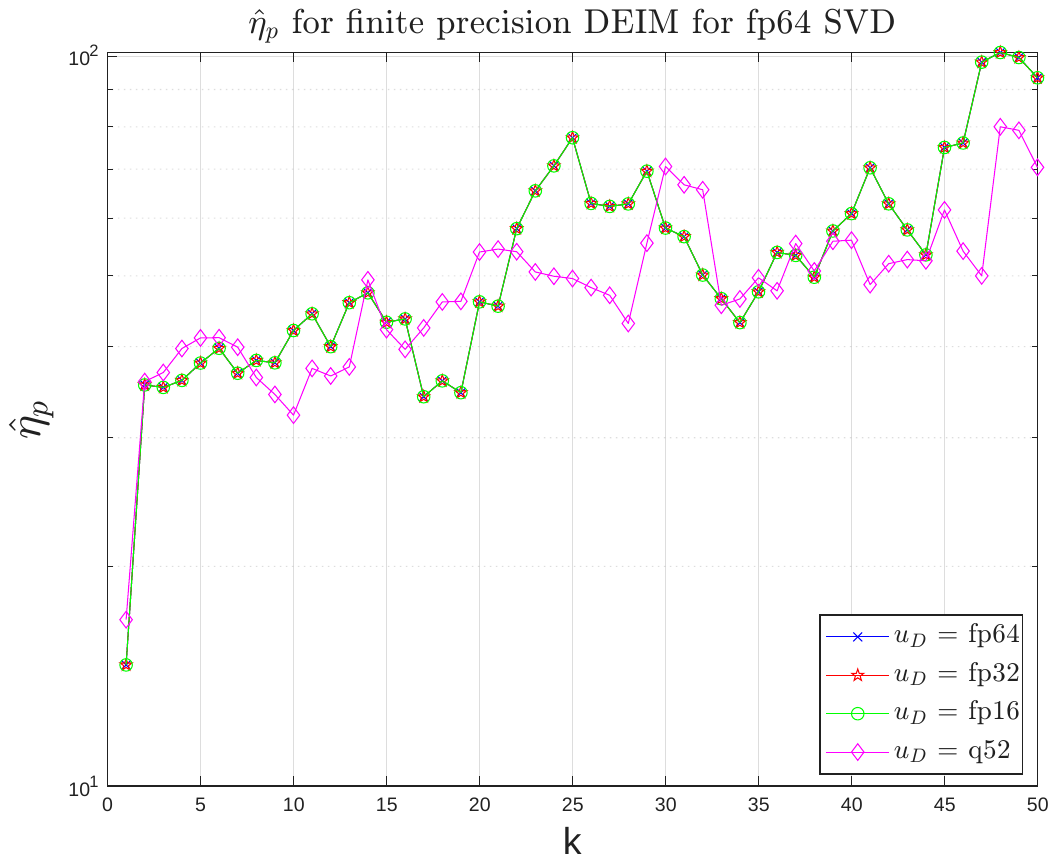} 
    \includegraphics[trim=1.5cm 6cm 1.5cm 6cm, clip,width=0.45\linewidth]{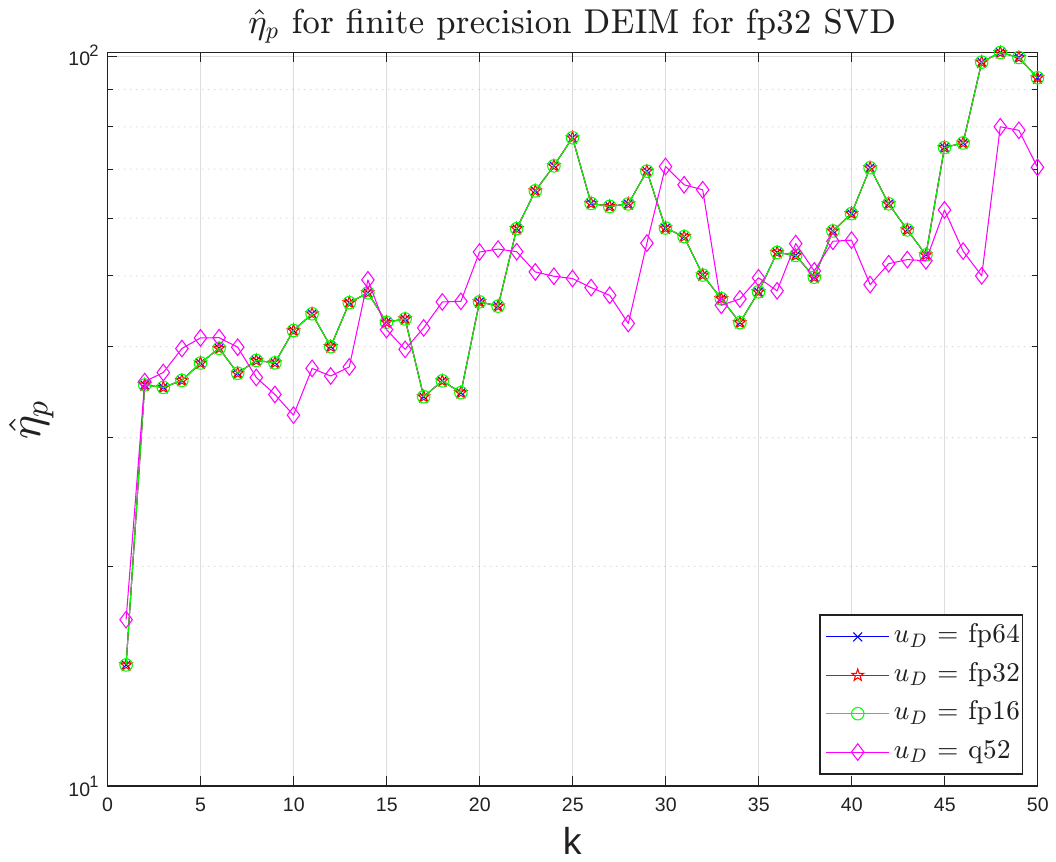} \\
    \includegraphics[trim=1.5cm 6cm 1.5cm 6cm, clip,width=0.45\linewidth]{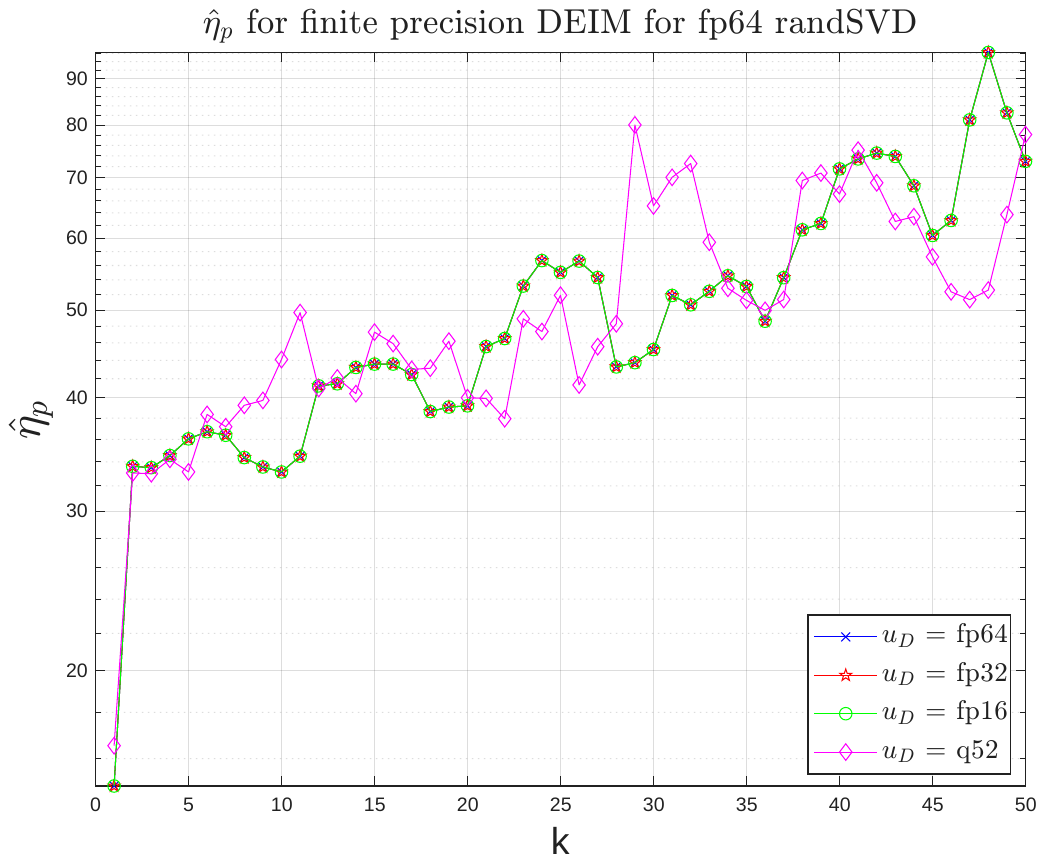} 
    \includegraphics[trim=1.5cm 6cm 1.5cm 6cm, clip,width=0.45\linewidth]{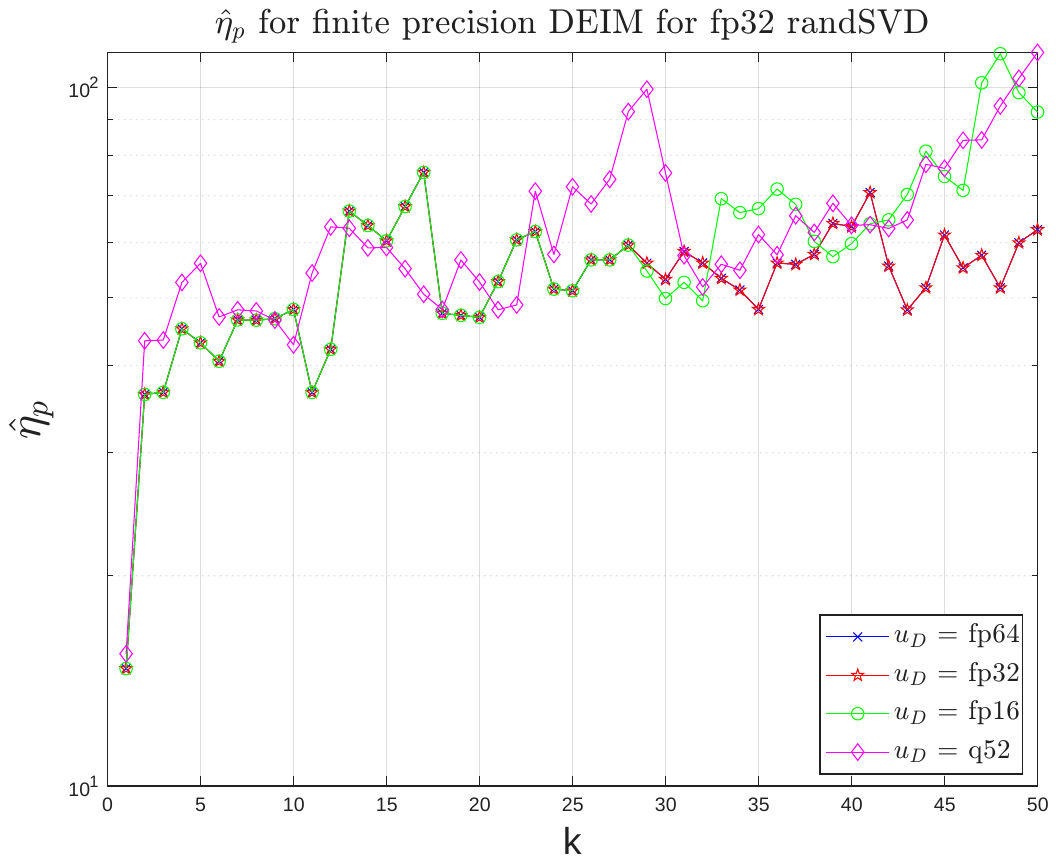} 
    \caption{Value of the amplification factor $\widehat{\eta}_p$ for DEIM in various precisions, where the input SVD is computed using the LAPACK routine (top row) and a randomized SVD with oversampling $p=10$ (bottom row), both computed in fp64 (left column) and fp32 (right column).}
    \label{fig:ex2eta}
\end{figure}

\subsection{Iterative SVD Computation in Low Precision}

\hspace{5mm}In this test case we reduce the precision of the SVD computation to fp16. We use a custom implementation of the Golub–Kahan–Lanczos bidiagonalization\footnote{\url{https://www.netlib.org/utk/people/JackDongarra/etemplates/node198.html}} with relative tolerance of $10^{-1}$, for a random starting vector and a maximum basis of 3 times the maximum number of $k$. We run experiments on both matrices from the previous experiments. Although we do not theoretically analyze this method in the paper to derive a condition analogous to \eqref{eq:usvdcond1} and \eqref{eq:usvdcond2}, Figures \ref{fig:ex3deim_mat1} and \ref{fig:ex3deim_mat2} show that even a very low precision SVD can be used while the approximation error and the amplification factor remain close to the higher-precision results.

\begin{figure}[t]
    \centering
    \includegraphics[trim=1.5cm 6cm 1.5cm 6cm, clip,width=0.45\linewidth]{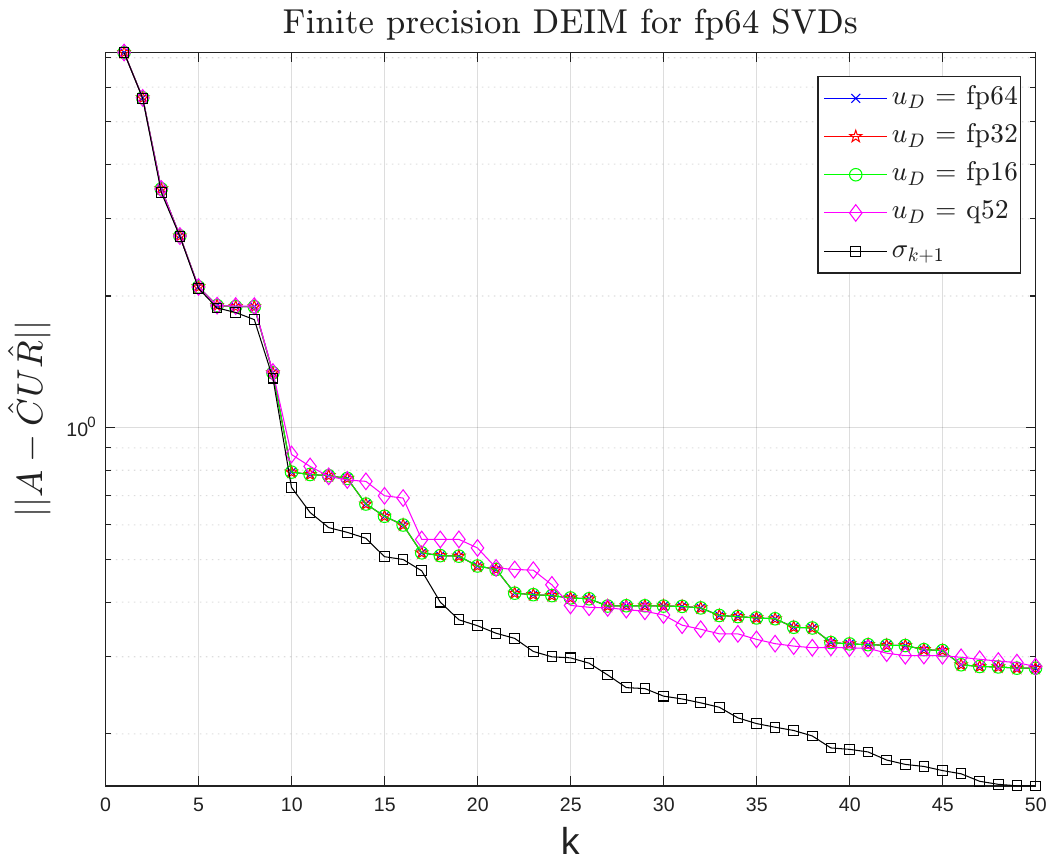} 
    \includegraphics[trim=1.5cm 6cm 1.5cm 6cm, clip,width=0.45\linewidth]{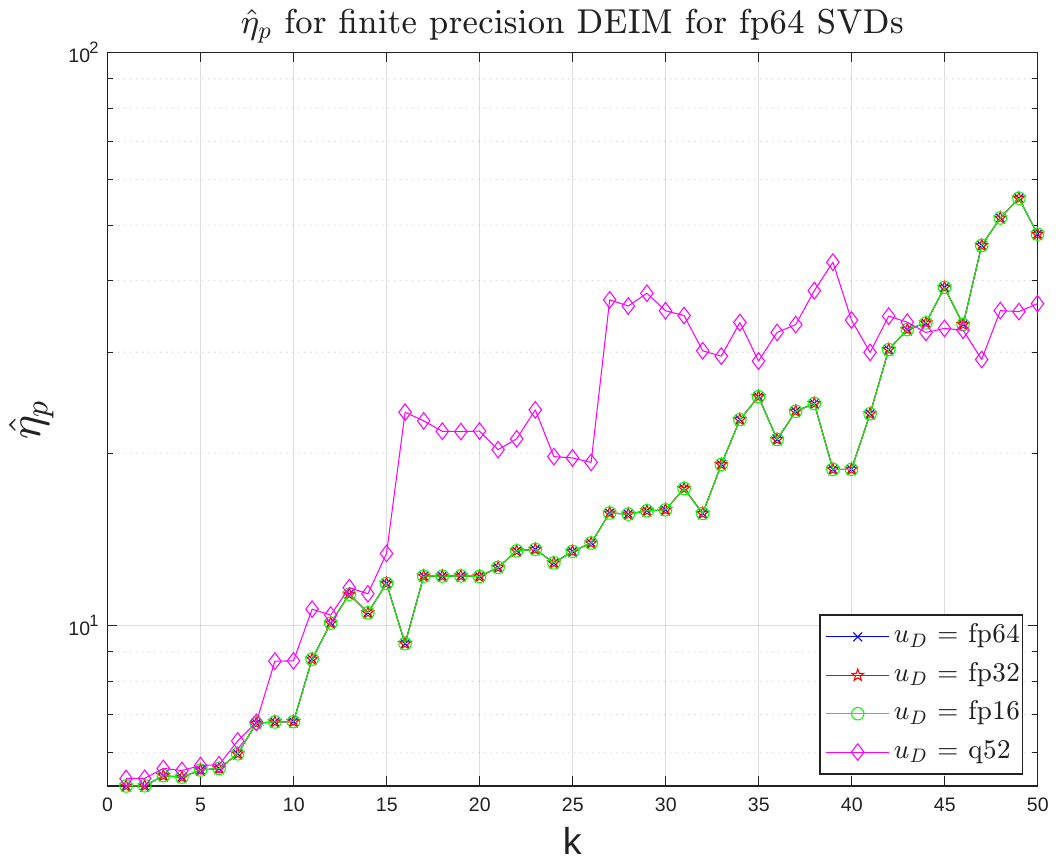} \\
    \includegraphics[trim=1.5cm 6cm 1.5cm 6cm, clip,width=0.45\linewidth]{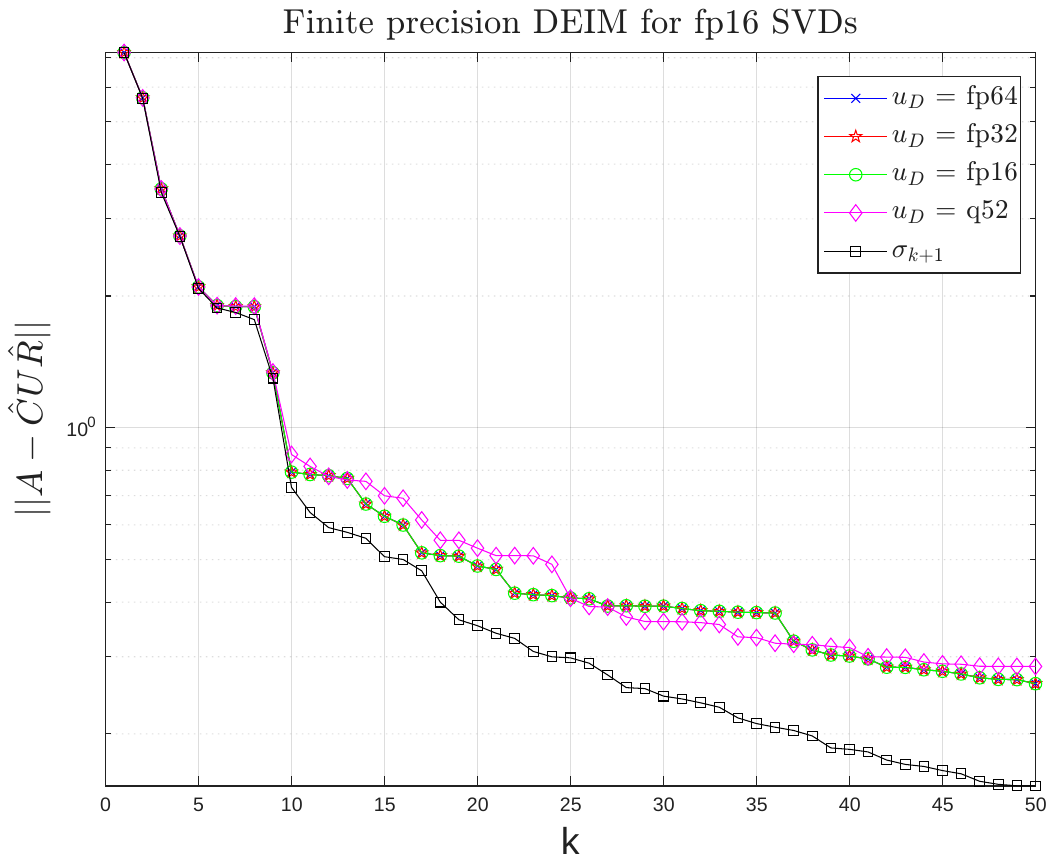} 
    \includegraphics[trim=1.5cm 6cm 1.5cm 6cm, clip,width=0.45\linewidth]{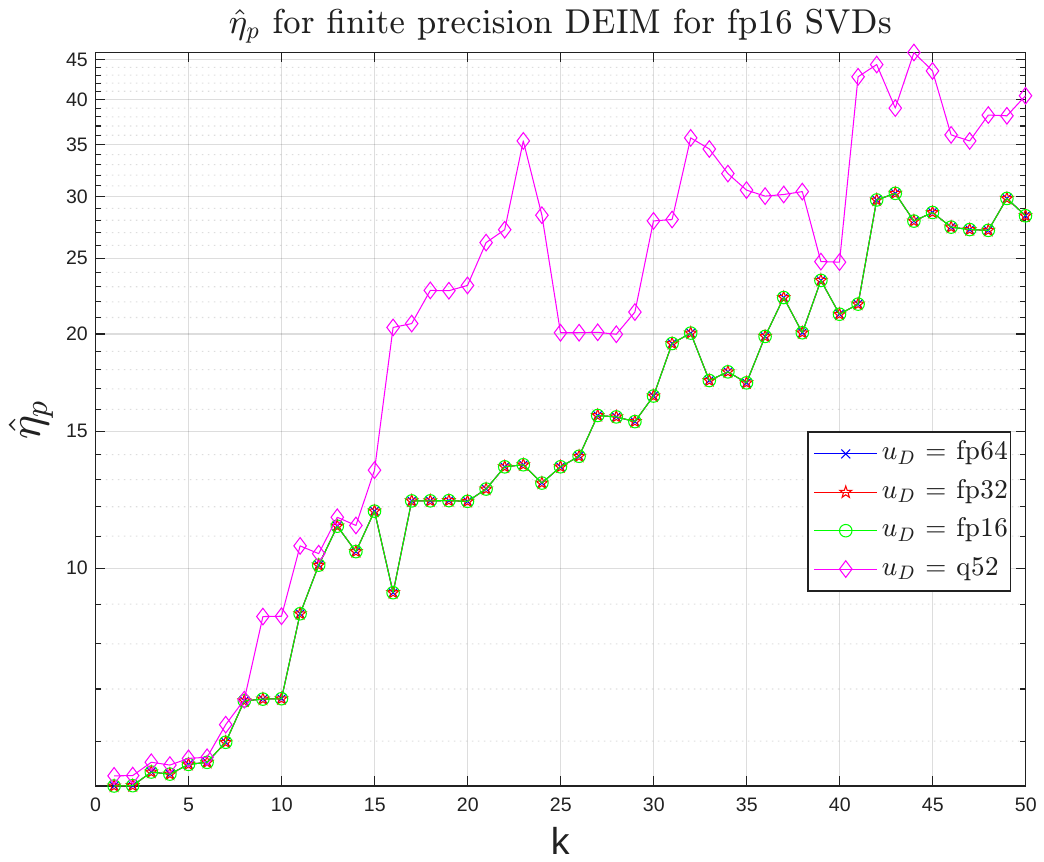} 
    \caption{For the matrix in Section \ref{section:exp1}, Error $\Vert A-\widehat{C}U\widehat{R}\Vert_2$ (left column)  and the value of the amplification factor $\widehat{\eta}_p$ (right column),  where the input SVD is computed using Golub-Kahan Lanczos bidiagonalization and DEIM is run in various precisions.}
    \label{fig:ex3deim_mat1}
\end{figure}

\begin{figure}
    \centering
    \includegraphics[trim=1.5cm 6cm 1.5cm 6cm, clip,width=0.45\linewidth]{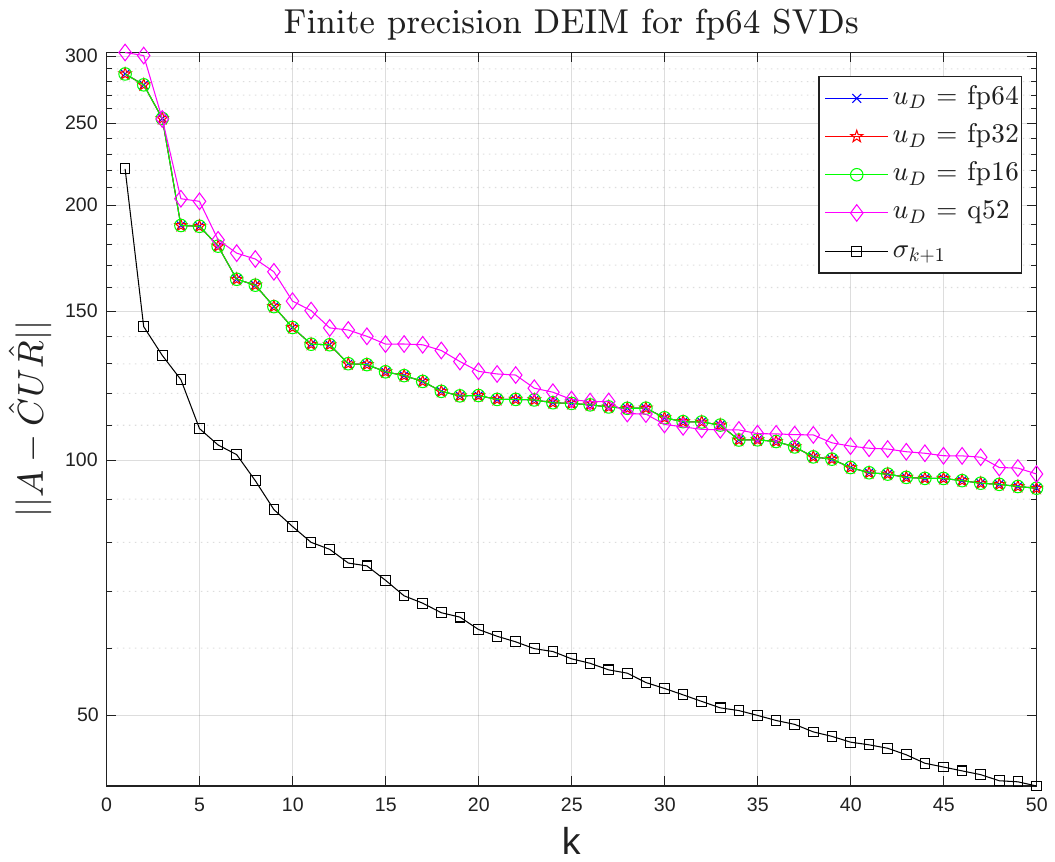} 
    \includegraphics[trim=1.5cm 6cm 1.5cm 6cm, clip,width=0.45\linewidth]{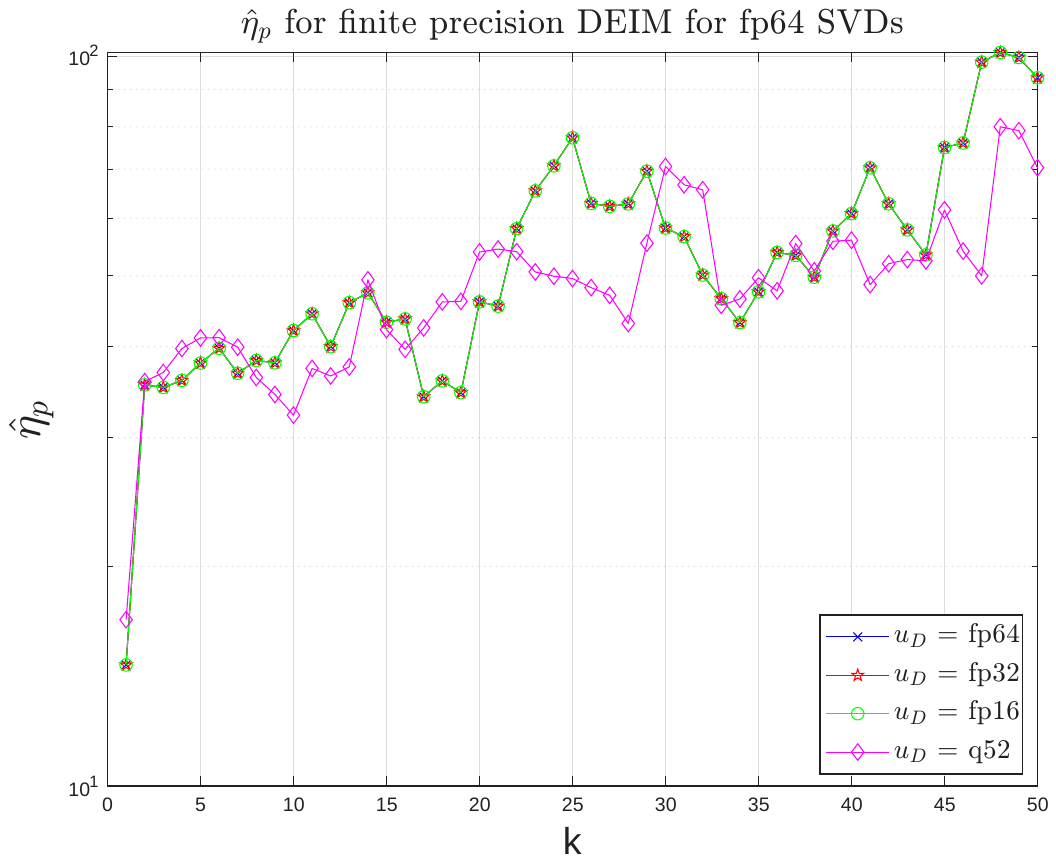} \\
    \includegraphics[trim=1.5cm 6cm 1.5cm 6cm, clip,width=0.45\linewidth]{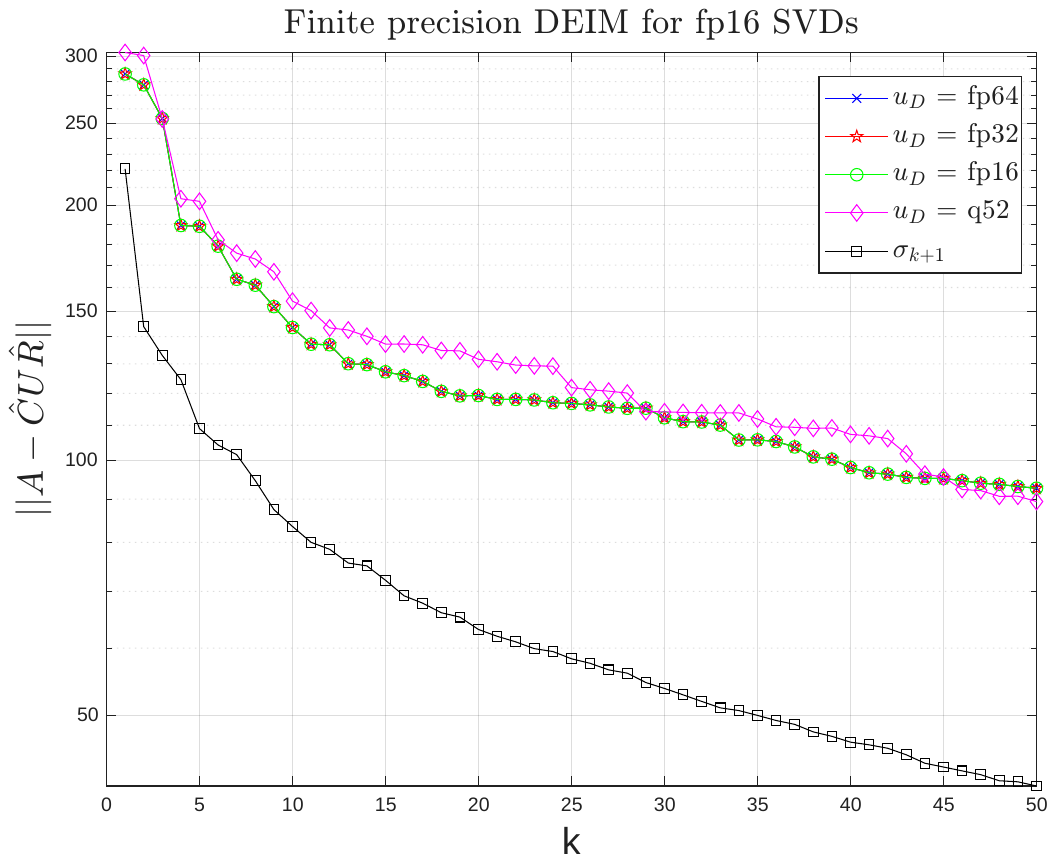} 
    \includegraphics[trim=1.5cm 6cm 1.5cm 6cm, clip,width=0.45\linewidth]{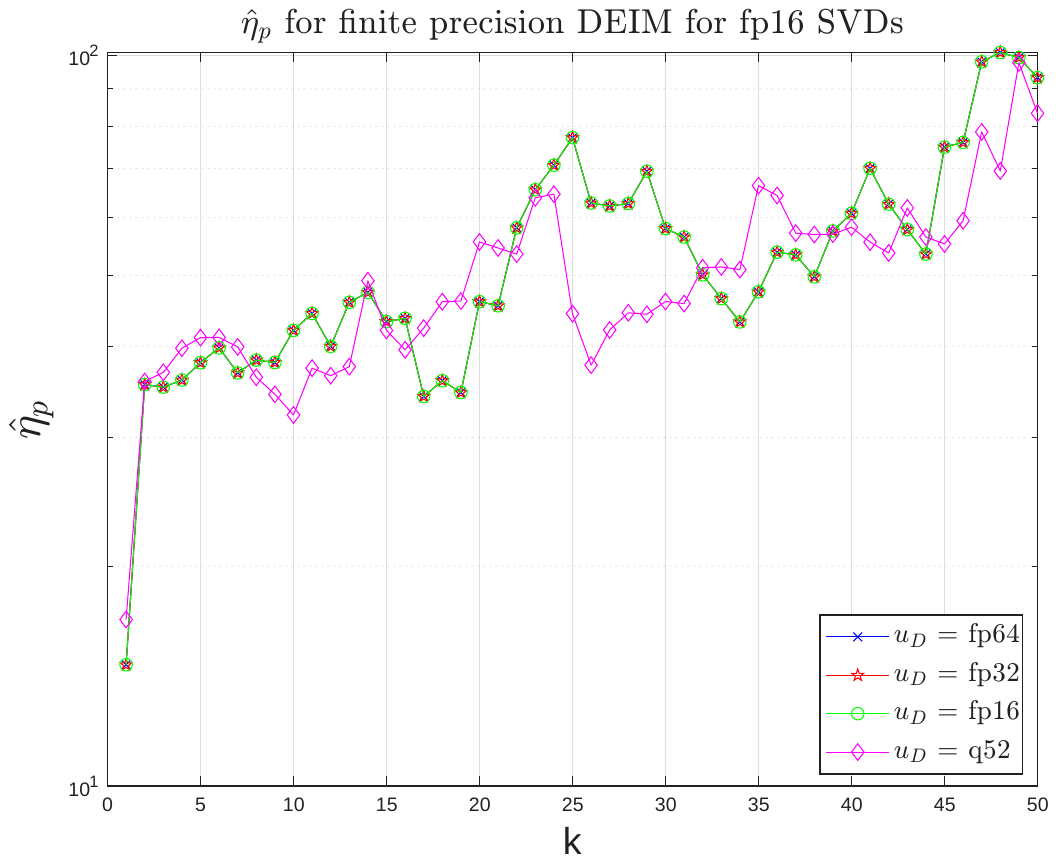} 
    \caption{For the matrix in Section \ref{section:exp2}, Error $\Vert A-\widehat{C}U\widehat{R}\Vert_2$ (left column)  and the value of the amplification factor $\widehat{\eta}_p$ (right column),  where the input SVD is computed using Golub-Kahan Lanczos bidiagonalization and for DEIM in various precisions.}
    \label{fig:ex3deim_mat2}
\end{figure}

\section{Conclusion}
\label{sec:conclusion}

In this work, we have investigated the effects of both the use of an inexact SVD and an inexact DEIM algorithm on the quality of DEIM-CUR low-rank approximations. 
We first prove a bound on the quality of a CUR factorization resulting from an inexactly computed SVD. This bound, analogous to the form proved in \cite{sorensen2016deim}, contains amplification factors which depend on the particular indices that are selected. We note that these bounds, presented in Section \ref{sec:inexactSVD}, are relevant to any of the index selection methods and therefore can be adapted to other index selection strategies. 

We then prove bounds on these amplification factors when indices are selected using a DEIM algorithm run in some precision $u_{D}$. Our bound \eqref{eq:uDEIM} suggests that in most practical cases, $u_{D}$ can be quite large (meaning that very low precision can be used) without significantly affecting the magnitude of the amplification factors for the computed indices. 

To demonstrate how the precision for the SVD, $u_S$, should be chosen, we take two examples of commonly-used SVD algorithms: LAPACK's \texttt{GESVD} and the randomized SVD with oversampling. In both cases, a clear heuristic emerges: the coarser the desired approximation (e.g., the lower the desired rank $k$), the lower precision we can likely use for the SVD computation; see \eqref{eq:usvdcond1} and \eqref{eq:usvdcond2}.  

Our numerical experiments demonstrate that there is very little difference in the resulting CUR approximation quality even when very low precisions are used for the SVD computation and the DEIM index selection. 

We note that here we have focused on the index selection part of the CUR approximation, where we have assumed both that the index selection is done by GEPP and that the computation of $U$ is performed exactly. While we suspect that it is necessary to compute $U$ precisely in order to maintain the quality of the approximation, quantifying the effects of error in the computation of $U$ and incorporating this into our analysis is an interesting further direction.

\bibliographystyle{alphadin}
\bibliography{paper}

\end{document}